\documentclass[a4paper,10pt]{amsart}

\usepackage[english]{babel}
\usepackage[utf8]{inputenc}

\usepackage{amssymb}
\usepackage{amsmath}
\usepackage{amsthm}
\usepackage{bbm}
\usepackage{enumerate}

\newcommand{\bbC}{\mathbb{C}}
\newcommand{\bbD}{\mathbb{D}}
\newcommand{\bbN}{\mathbb{N}}
\newcommand{\bbQ}{\mathbb{Q}}
\newcommand{\bbR}{\mathbb{R}}
\newcommand{\bbT}{\mathbb{T}}
\newcommand{\bbZ}{\mathbb{Z}}
\newcommand{\calK}{\mathcal{K}}
\newcommand{\calL}{\mathcal{L}}
\newcommand{\calS}{\mathcal{S}}
\newcommand{\calU}{\mathcal{U}}

\newcommand{\per}{\operatorname{per}}
\newcommand{\pnt}{\operatorname{pnt}}

\theoremstyle{definition}
\newtheorem{definition}{Definition}[section]
\newtheorem{remark}[definition]{Remark}

\newtheorem{example}[definition]{Example}
\newtheorem{examples}[definition]{Examples}

\theoremstyle{plain}
\newtheorem{proposition}[definition]{Proposition}
\newtheorem{lemma}[definition]{Lemma}
\newtheorem{theorem}[definition]{Theorem}
\newtheorem{corollary}[definition]{Corollary}

\newtheorem*{theorem_no_number}{Theorem}

\begin{document}

\title{On the Peripheral Spectrum of Positive Operators}

\thanks{During his work on this article the author was supported by a scholarship of the ``Landesgraduiertenf\"orderung Baden-W\"urttemberg''.}

\author{Jochen Gl\"uck}

\address{Institute of Applied Analysis, Ulm University, 89069 Ulm, Germany}
\email{jochen.glueck@uni-ulm.de}

\begin{abstract}
	This paper contributes to the analysis of the peripheral (point) spectrum of positive linear operators on Banach lattices. We show that, under appropriate growth and regularity conditions, the peripheral point spectrum of a positive operator is cyclic and that the corresponding eigenspaces fulfil a certain dimension estimate. A couple of examples demonstrates that some of our theorems are optimal. Our results on the peripheral point spectrum are then used to prove a sufficient condition for the peripheral spectrum of a positive operator to be cyclic; this generalizes theorems of Lotz and Scheffold. 
\end{abstract}


\maketitle

\section{Introduction and Preliminaries} \label{section_introduction}

The study of spectral properties of positive operators in infinite dimensions is by now a classical topic in operator theory. Beginning with the results of Kre{\u\i}n and Rutman in \cite{Kreuin1950}, many interesting and sophisticated spectral results on those operators have been proved. Positive operators on Banach lattices are particularly well-behaved in this context; we refer to the survey article \cite{Grobler1995} for an overview of their spectral theory. Still today, the field raises the interest of many researches; recent contributions to the theory concern for example the analysis of essential spectra of positive operators \cite{Alekhno2007}, \cite{Alekhno2009} and new comparison results on positive operators and there spectral radius \cite{Gao2013}. \par 

In this paper we are mainly concerned with the \emph{peripheral spectrum} $\sigma_{\per}(T) = \sigma(T) \cap \{z \in \bbC: \, |z| = r(T)\}$ and the \emph{peripheral point spectrum} $\sigma_{\per,\pnt}(T) := \sigma_{\pnt}(T) \cap \{z \in \bbC: \, |z| = r(T)\}$ of a positive operator $T$ on a complex Banach lattice $E$; here, $r(T)$ denotes the spectral radius of $T$, $\sigma(T)$ is its spectrum and $\sigma_{\pnt}(T)$ its point spectrum. Recall that a set $M \subset \bbC$ is called \emph{cyclic} if $re^{i\theta} \in M$ ($r\ge 0$, $\theta \in \bbR$) implies that $re^{in\theta} \in M$ for each $n \in \bbZ$. An interesting question in the spectral theory of positive operators is to find conditions which ensure that the peripheral spectrum or the peripheral point spectrum of a positive operator is cyclic. \par 

Our goal is to give several new such conditions. To do so we first prove some results on the fixed space of a Markov operator in Section~\ref{section_fixed_space_markov} and use them for a preliminary analysis of certain eigenvalues of positive operators in Section~\ref{section_eigenvalues_with_dom_eigenvectors}. In Section~\ref{section_ws_bounded_operators} we then introduce the notion of a \emph{(WS)-bounded} operator on a complex Banach space. This is a rather weak boundedness condition whose exact definition requires some preliminary work; for a first impression about this notion, we refer the reader to Examples~\ref{examp_ws_bounded} where several classes of (WS)-bounded operators are listed. For example, an operator $T$ (with $r(T) = 1$) is (WS)-bounded if it is Abel-bounded or if it fulfils the condition $\liminf_{n \to \infty} ||T^n|| < \infty$. In Section~\ref{section_per_pnt_spec_general} we are finally able to give some sufficient conditions for the peripheral point spectrum of a positive operator to be cyclic. We sum up our main results from this section in the following theorem:

\begin{theorem_no_number}
	Let $T$ be a positive operator on a complex Banach lattice $E$, $r(T) = 1$, and assume that at least one of the following conditions is fulfilled:
	\begin{enumerate}[(a)]
		\item $E$ has a pre-dual Banach lattice, $T$ has a pre-adjoint and $T$ is (WS)-bounded. \par 
		\item $E$ is a KB-space and $T$ is (WS)-bounded. \par 
		\item $E$ has order-continuous norm and $T$ is mean ergodic. \par 
		\item $T$ is weakly almost-periodic.
	\end{enumerate}
	Then we have $\dim \ker(e^{i\theta}-T) \le \dim \ker(e^{in\theta}-T)$ for each $n \in \bbZ$ and each $\theta \in \bbR$. In particular, the peripheral point spectrum of $T$ is cyclic.
\end{theorem_no_number}

In the above theorem we use the convention that the dimension of a vector space is either an integer or $\infty$, i.e.~we do not distinguish between different infinite cardinalities. For the notions used in the conditions (a)--(d) we refer to Section~\ref{section_per_pnt_spec_general}. The above theorem contains generalizations of several known results; details and references are also given in Section~\ref{section_per_pnt_spec_general}. In Section~\ref{section_per_pnt_spec_markov} we continue our analysis of the peripheral point spectrum of positive operators with special emphasis on Markov operators; see the end of the introduction for a definition of these operators. \par 

In Section~\ref{section_per_spec} we deal with the peripheral spectrum of positive operators. It is a long open question whether every positive operator on a complex Banach lattice has cyclic peripheral spectrum. We are not able to solve this here, but we present a generalization of some known results. More precisely, we prove the following result in Theorem~\ref{thm_ws_bounded_op_zykl_per_spec}:

\begin{theorem_no_number}
	Let $T$ be a positive operator on a complex Banach lattice $E$, $r(T) = 1$. If $T$ is (WS)-bounded, then the peripheral spectrum of $T$ is cyclic.
\end{theorem_no_number}

This generalizes a result of Lotz \cite[Theorem 4.7]{Lotz1968} who proved that every Abel-bounded positive operator has cyclic peripheral spectrum and a result of Scheffold \cite[Satz~3.6]{Scheffold1971} who proved that a positive operator $T$ with spectral radius $1$ has cyclic peripheral spectrum if $\liminf_{n \to \infty} ||T^n|| < \infty$. For further contributions to this topic we refer the reader to Lotz \cite[Theorem 4.9 and Theorem 4.10]{Lotz1968} (see also \cite[Theorem V.4.9 and its Corollary]{Schaefer1974} for an English version), to Krieger \cite[Satz~2.2.3]{Krieger1969} (see also \cite[p.~352]{Schaefer1974} where this result is stated in English) and to Zhang \cite[Theorem~2.11]{Zhang1993}. In Section~\ref{section_boundary_pnt_spec_of_markov_sg} we briefly deal with $C_0$-semigroups of Markov operators and we show that one of our examples in Section~\ref{section_per_pnt_spec_markov} can be adapted to the $C_0$-semigroup case. In the appendix, we recall some facts about the signum operator on complex Banach lattices. \par 

Throughout the article, the reader is assumed to be familiar with the theory of (real and complex) Banach lattices and with the construction of filter (and, in particular, ultra) products of Banach lattices; for the latter topic see for example \cite[Section~V.1]{Schaefer1974} or \cite[p.\,251--253]{Meyer-Nieberg1991}. To read Section~\ref{section_boundary_pnt_spec_of_markov_sg}, familiarity with the basic concepts of $C_0$-semigroups is also required. We refer to \cite{Engel2000} as a standard reference for this topic. \par 

Let us fix some notation: If $X$ is a real or complex Banach space, then we denote by $\calL(X)$ the space of bounded linear operators on $X$. If $X$ is a complex Banach space and $T \in \calL(X)$, then we denote by $\sigma(T)$ the \emph{spectrum} and by $\sigma_{\operatorname{pnt}}(T)$ the \emph{point spectrum} of $T$. If $\lambda \in \bbC \setminus \sigma(T)$, then $R(\lambda,T) := (\lambda - T)^{-1}$ denotes the \emph{resolvent} of $T$ in $\lambda$. The same notations are also used for the (point) spectrum and the resolvent of unbounded operators $A: X \supset D(A) \to X$. If $E$ is a complex Banach lattice, then it is by definition the complexification of a real Banach lattice $E_\bbR$ (cf.~\cite[Section~II.11]{Schaefer1974}); if $x,y \in E$, then assertions such as $x \ge y$ are always to be understood as a shorthand for $x,y \in E_\bbR$ and $x \ge y$. If $E$ is a real or complex Banach lattice and $x,y \in E$, then we write $x > y$ to say that $x \ge y$, but $x \not= y$. If $K$ is a compact Hausdorff space, then the space of real (respectively complex) valued continuous functions on $K$ is denoted by $C(K;\bbR)$ (respectively by $C(K;\bbC)$); those spaces will always be endowed with the supremum norm $||\cdot||_\infty$. A bounded linear operator $T$ on $C(K;\bbR)$ or $C(K;\bbC)$ is called a \emph{Markov operator} if $T$ is positive and if $T\mathbbm{1} = \mathbbm{1}$. By $c(\bbN;\bbC)$ we denote the space of all complex-valued convergent sequences which are indexed by $\bbN$; by $c_0(\bbN;\bbC) \subset c(\bbN;\bbC)$ we denote the space of all complex-valued sequences which converge to $0$. The corresponding spaces of real-valued sequences are denoted by $c(\bbN;\bbR)$ and $c_0(\bbN;\bbR)$, respectively. A vector subspace $V$ of a real Banach lattice $E$ is called a \emph{lattice subspace} of $E$ if $V$ is a vector lattice with respect to the order induced by $E$; $V$ is called a \emph{sublattice} of $E$ if $|x| \in V$ for every $x \in V$. We use the symbol $\bbT := \{z \in \bbC: |z| = 1\}$ to denote the complex unit circle. Further notation is introduced as it is needed.

\section{The fixed space of Markov operators} \label{section_fixed_space_markov}

The purpose of this section is to establish the following structure result on the fixed space of a Markov operator on an order-complete $C(K,\bbR)$-space; will shall need this result in the subsequent sections. For the definition of the notion \emph{AM-space} we refer the reader to \cite[Definition~II.7.1]{Schaefer1974}.

\begin{theorem} \label{thm_fixed_space_of_markov_op_real}
	Let $E$ be an order complete $C(K;\bbR)$-space and let $T$ be a Markov operator on $E$. Then the fixed space $F := \ker(1-T)$ is an order complete lattice subspace (not necessarily a sublattice) of $E$. When endowed with the supremum norm $||\cdot||_\infty$, $F$ is a Banach lattice and an AM-space with unit $\mathbbm{1}$.
\end{theorem}
\begin{proof}
	Let $\emptyset \not= G \subset F$ be bounded above in $F$; we have to show that $G$ has a supremum in $F$. To do so, first note that $c := \sup\{||g||_\infty: g \in G\} < \infty$ and that $G$ is in fact bounded above by $c\mathbbm{1}$. Now, consider the set
	\begin{align*}
		A := \{h \in E: & \; h \le Th \text{, } h \text{ is an upper bound of } G \text{ and} \\
		& h \le k \text{ for every } k \in F \text{ which is an upper bound of } G\} \text{.}
	\end{align*}
	Then the supremum of $G$ in $E$ is contained in $A$, so $A$ is non-empty. Moreover, $A$ is bounded above by $c\mathbbm{1}$. Furthermore, $A$ is clearly $T$-invariant and whenever $\emptyset \not= B \subset A$, then $\sup B$ exists in $E$ and is contained in $A$. Now, let $g_{\max} := \sup A$. Then $g_{\max} \in A$, and we thus have $g_{\max} \le Tg_{\max} \in A$. By definition of $g_{\max}$ this implies that $Tg_{\max} = g_{\max}$, so $g_{\max} \in F$. Since $g_{\max} \in A$, we conclude that $g_{\max}$ is the supremum of $G$ in $F$. This shows that $F$ is a lattice subspace of $E$ and that it is order complete. Besides this, note that we have $g_{\max} \le c \mathbbm{1}$. \par 		
	Let us show next that $F$ is a Banach lattice with respect to $||\cdot||_\infty$. If $f \in F$, then the modulus $|f|_F$ of $f$ in $F$ is the supremum of $\pm f$ in $F$; hence, $|f|_F$ is given by the function $g_{\max}$ above if we choose $G = \{f,-f\}$. In this case, $c = ||f||_\infty$ and we thus have $|f| \le |f|_F = g_{\max} \le ||f||_\infty \mathbbm{1}$. This shows that $||\,|f|_F||_\infty = ||f||_\infty$. Since we clearly have $||f_1||_\infty \le ||f_2||_\infty$ for all $f_1,f_2 \in F$ with $0 \le f_1 \le f_2$, we conclude that $F$ is indeed a Banach lattice with respect to the $||\cdot||_\infty$-norm. \par 
	To show that $F$ is an AM-space with respect to the norm $||\cdot||_\infty$, let $0 \le g_1,g_2 \in F$. If $G = \{g_1,g_2\}$, then the supremum $g_1 \lor_F g_2$ of $g_1$ and $g_2$ in $F$ is given by the function $g_{\max}$ above. Since $C(K;\bbR)$ is an AM-space, we have $c = ||g_1||_\infty \lor \, ||g_2||_\infty = ||g_1 \lor g_2||_\infty$ and hence
	\begin{align*}
		g_1 \lor g_2 \, \le \, g_1 \lor_F g_2 = g_{\max} \le ||g_1 \lor g_2||_\infty \mathbbm{1} \text{.}
	\end{align*}
	Therefore, $||g_1 \lor_F g_2||_\infty = ||g_1 \lor g_2||_\infty = ||g_1||_\infty \lor ||g_2||_\infty$. Thus, $F$ is an AM-space with respect to the norm $||\cdot||_\infty$. Moreover, $\mathbbm{1}$ is an element of $F$; since $\mathbbm{1}$ is the largest element in the unit ball of $C(K,\bbR)$, it is in particular the largest element of the unit ball in $F$. Hence, the AM-space $F$ contains $\mathbbm{1}$ as a unit.  
\end{proof}

\begin{corollary} \label{cor_fixed_space_of_markov_op_real}
	Let $E$ be an order complete $C(K;\bbC)$-space, let $T$ be a Markov operator on $E$ and denote its fixed space by $F = \ker(1-T)$. \par 
	If $F_\bbR := F \cap C(K;\bbR)$, then we have $F = F_\bbR + iF_\bbR$ and therefore, $F$ is a complex order-complete AM-space with unit $\mathbbm{1}$ when endowed with an appropriate equivalent norm.
\end{corollary}
\begin{proof}
	Since $T$ maps real-valued functions to real-valued functions we have $F = F_\bbR + i F_\bbR$. The remaining assertions follow from Theorem~\ref{thm_fixed_space_of_markov_op_real}. 
\end{proof}

In view of Theorem~\ref{thm_fixed_space_of_markov_op_real} it is natural to ask whether the fixed space of the Markov operator $T$ is in fact a sublattice of $C(K;\bbR)$; moreover, if we drop the condition on $C(K;\bbR)$ to be order complete, one might ask whether the fixed space of $T$ is, though no longer order complete, still a lattice subspace of $C(K;\bbR)$. The answer to both questions is negative as the following two examples show.

\begin{examples} \label{examp_fixed_space_of_markov_op}
	(a) There is a Markov operator $T$ on $\bbR^3$ such that $F := \ker(1-T)$ is not a sublattice of $\bbR^3$. \par 
	Indeed, let $T$ be the operator whose representation matrix with respect to the canonical basis is given by
	\begin{align*}
		\begin{pmatrix}
			1 & 0 & 0 \\
			\frac{1}{3} & \frac{1}{3} & \frac{1}{3} \\
			0 & 0 & 1
		\end{pmatrix} \text{.}
	\end{align*}
	Then $T$ is clearly a Markov operator and its fixed space $\ker(1-T)$ coincides with the linear span of $(1,1,1)$ and $(1,0,-1)$. Now, consider the vector $\hat f = (1,0,-1) \in \ker(1-T)$. 
	Then the supremum of $\pm \hat f$ in $\bbR^3$ is given by $|\hat f| = (1,0,1)$, but this vector is not contained in the fixed space $\ker(1-T)$. Hence, $\ker(1-T)$ is not a sublattice of $\bbR^3$. However, it follows from Theorem~\ref{thm_fixed_space_of_markov_op_real} that it is a lattice subspace of $\bbR^3$, so $\pm \hat f$ must have a supremum in $\ker(1-T)$ (and it is easy to check that this supremum is given by $(1,1,1)$). \par
	(b) There is a (non order-complete) $C(K;\bbR)$-space $E$ and a Markov operator $S$ on $E$ such that the fixed space $\ker(1-S)$ is not a lattice subspace of $E$. Indeed, let
	\begin{align*}
		E  & = \{(f,g,h) \in \bbR^3 \times c(\bbN; \bbR) \times c(\bbN; \bbR): \lim g = \lim h\} \text{,}
	\end{align*}
	where $c(\bbN;\bbR)$ is the space of all real-valued convergent sequences. Clearly, $E$ is an AM-space with unit and thus isometrically lattice isomorphic to some $C(K;\bbR)$-space. Let $T \in \calL(\bbR^3)$ be the operator from Example (a) and define $S \in \calL(E)$ by $S(f,g,h) = (f',g',h')$ where
	\begin{align*}
		f' = Tf \text{,} \qquad g' = (f_2,g_1,g_2,g_3,...) \text{,} \qquad h' = h \text{.}
	\end{align*}
	Note that $S$ indeed maps $E$ into $E$, and that $S$ is a Markov operator on $E$. The fixed space of $S$ is given by
	\begin{align*}
		\ker(1-S) = \{(f,g,h) \in E: f \in \ker(1-T) \text{, } g = f_2 \mathbbm{1}_\bbN \} \text{.}
	\end{align*}
	As in Example (a) let $\hat f = (1,0,-1) \in \ker(1-T)$; we now consider the element $(\hat f,0,0) \in \ker(1-S)$ and show that $\pm (\hat f,0,0)$ does not have a supremum in $\ker(1-S)$. \par 
	Indeed, assume for a contradiction that $(f',g',h') \in \ker(1-S)$ is the lowest upper bound of $\pm(\hat f,0,0)$ in $\ker(1-S)$. Then it follows from Example (a) that $f' \ge (1,1,1)$ and hence $g' = f_2 \mathbbm{1}_\bbN \ge \mathbbm{1}_\bbN$. The vector $h' \in c(\bbN;\bbR)$ has to fulfil $0 \le h'$ and $\lim h' = \lim g' \ge 1$. Due to this estimate, $h'$ is non-zero, so we can find another element $0 \le h'' \in c(\bbN;\bbR)$ which fulfils $h'' < h'$ as well as $\lim h'' = \lim g'$. Hence, $(f',g',h'') \in \ker(1-S)$ is also an upper bound of $\pm(f,0,0)$ but it is smaller then $(f',g',h')$. This is a contradiction.
\end{examples}

\section{Eigenvalues with dominated eigenvectors} \label{section_eigenvalues_with_dom_eigenvectors}

In this section we prove cyclicity results for eigenvalues whose corresponding eigenvectors satisfy certain domination properties; we also give estimates on the dimensions of the corresponding eigenspaces. To do so, we need the following proposition, which is based on some well-known facts from Perron-Frobenius theory.

\begin{proposition} \label{prop_eigenvalues_and_funtions_of_lattice_hom}
	Let $E$ be an order complete complex Banach lattice and let $T$ be a lattice homomorphism on $E$. 
	\begin{enumerate}[(a)]
		\item If $e^{i\theta}$ ($\theta \in \bbR$) is an eigenvalue of $T$ with corresponding eigenvector $z \not= 0$, then $|z|$ is an eigenvector of $T$ for the eigenvalue $1$. \par 
		\item Let $\theta \in \bbR$. Then we have
			\begin{align*}
				\dim \ker(e^{i\theta} - T) \le \dim \ker(e^{in\theta} - T)
			\end{align*}
			for every $n \in \bbZ$.
	\end{enumerate}
\end{proposition}

For the proof we need the vectors $f^{[n]}$ which are given in Definition~\ref{def_lattice_powers_of_a_function} for every non-zero element $f$ of a complex Banach lattice. 

\begin{proof}[Proof of Proposition~\ref{prop_eigenvalues_and_funtions_of_lattice_hom}]
	Assertion (a) is obvious. To prove (b), let $0 \not= z \in \ker(e^{i\theta}-T)$. When endowed with an appropriate norm, the principal ideal $E_{|z|}$ is an AM-space with unit $|z|$ (see \cite[the corollary of Proposition~II.7.2]{Schaefer1974}). It follows from (a) that $E_{|z|}$ is $T$-invariant and that $T|_{E_{|z|}}$ is a Markov operator on $E_{|z|}$. We can thus conclude from \cite[the last sentence of Proposition V.4.2]{Schaefer1974} that $z^{[n]} \in \ker(e^{in\theta}-T)$ for each $n \in \bbZ$. Lemma~\ref{lem_dimension_estimate_complex_powers} now immediately yields assertion (b).  
\end{proof}

Next we will prove a new structure result on the point spectrum of certain positive operators. We need the following notation: A Banach lattice $E$ is said to have a \emph{pre-dual} Banach lattice if there exists a Banach lattice $F$ such that $F' = E$. If this is the case and if $T \in \calL(E)$, then we say that $T$ has a \emph{pre-adjoint} if there is an operator $S \in \calL(F)$ such that $S' = T$. Note that if $T$ is positive, then so is its pre-adjoint $S$. 

\begin{theorem} \label{thm_dominated_eigenvector_dual}
	Let $T$ be a positive operator on a complex Banach lattice $E$. Suppose that $E$ has a pre-dual Banach lattice and that $T$ has a pre-adjoint. Let $r > 0$, $\theta \in \bbR$ and let $0 < x \in \ker(r-T)$. Then we have
	\begin{align*}
		\dim[E_x \cap \ker(re^{i\theta}-T)] \le \dim[E_x \cap \ker(re^{in\theta}-T)]
	\end{align*}
	for all $n \in \bbZ$. In particular, if $re^{i\theta}$ is an eigenvalue of $T$ with eigenvector $0\not= z \in E_x$, then $re^{in\theta}$ is also an eigenvalue of $T$ with an eigenvector in $E_x$ for each $n \in \bbZ$.
\end{theorem}

Note that in this theorem we do not make any assumption on the spectral radius of $T$, i.e.~the theorem also holds if $r < r(T)$. We point out that the major advance in the theorem is the rather weak relation that we require between the eigenvector $z$ for the eigenvalue $re^{i\theta}$ and the eigenvector $x$ for the eigenvalue $r$. In the classical approach to Perron-Frobenius theory which can for example be found in \cite[Sections~V.4 and V.5]{Schaefer1974}, it is usually required that $|z| \in \ker(r-T)$ for some $z \in \ker(re^{i\theta}-T)$ to conclude that the numbers $re^{in\theta}$ ($n \in \bbZ$) are also eigenvalues of $T$. In the above theorem, it suffices that there are eigenvectors $0\not= z \in \ker(re^{i\theta}-T)$ and $0<x \in \ker(r-T)$ which fulfil the domination property $|z| \le x$. \par 
The following proof of Theorem~\ref{thm_dominated_eigenvector_dual} as well as many of its subsequent applications were inspired by \cite[Corollary~C-III.4.3]{Arendt1986}.

\begin{proof}[Proof of Theorem~\ref{thm_dominated_eigenvector_dual}]
	We may assume that $r=1$. With respect to an appropriate norm the principal ideal $E_x$ is an AM-space with unit $x$ (see \cite[the corollary of Proposition~II.7.2]{Schaefer1974}). The operator $T$ leaves $E_x$ invariant, and the restriction $T|_{E_x}$ is a Markov operator on $E_x$. Now, choose a sequence $(m_k)_{k \in \bbN}$ of integers $m_k \ge 2$ such that $e^{im_k\theta} \to 1$ and fix a free ultra filter $\calU$ on $\bbN$. For each $f \in E_x$, the sequence $(T^m f)_{m \in \bbN_0}$ is bounded in $E_x$ and thus in $E$. Therefore, the limits
	\begin{align*}
		Rf := w^*_E\text{-}\lim_\calU T^{m_k - 1}f \quad \text{and} \quad Sf := w^*_E\text{-}\lim_\calU T^{m_k}f
	\end{align*}
	exist for each $f \in E_x$, where $w^*_E\text{-}\lim_\calU$ denotes the limit in the weak${}^*$-topology on $E$ with respect to $\calU$. Moreover, both limits are again located in $E_x$. Thus, $R$ and $S$ are linear operators on the AM-space $E_x$. Clearly, $R$ and $S$ are Markov operators on $E_x$, and $E_x \cap \ker(e^{i\theta}-T)$ is contained in the fixed space of $S$ since for every $z \in E_x \cap \ker(e^{i\theta}-T)$ we have $T^{m_k}z = e^{im_k\theta}z \overset{\calU}{\to} z$ even with respect to the norm topology on $E$. \par
	Moreover, we have $R\,T|_{E_x} = T|_{E_x}R = S$, and the operator $S$ commutes with $T|_{E_x}$ and $R$; simply use the continuity of $T$ with respect to the weak${}^*$-topology to see that those assertions hold. Hence, $T|_{E_x}$ and $R$ leave the fixed space $F := \ker(1-S)$ invariant, and the restrictions $T|_F$ and $R|_F$ are inverse to each other. \par
	We know from Corollary~\ref{cor_fixed_space_of_markov_op_real} that $F$ is a complex Banach lattice with respect to an appropriate norm. The operator $T|_F$ and its inverse $R|_F$ are positive, so we conclude that $T|_F$ is a lattice isomorphism. For each $n\in\bbZ$, Proposition~\ref{prop_eigenvalues_and_funtions_of_lattice_hom} now yields that
	\begin{align*}
		\dim(E_x \cap \ker(e^{i\theta}-T)) & = \dim \ker(e^{i\theta}-T|_F) \le \\
		& \le \dim \ker(e^{in\theta}-T|_F) \le \dim[E_x \cap \ker(e^{in\theta}-T)] \text{.}
	\end{align*}
	This proves the assertion.  
\end{proof}

Recall that a Banach lattice $E$ can always be considered a sublattice of its bi-dual $E''$ by means of evaluation, see \cite[Corollary 2 of Proposition II.5.5]{Schaefer1974}. If $E$ has order-continuous norm, then it is even an ideal in $E''$, see \cite[Theorem~2.4.2]{Meyer-Nieberg1991}. This observation yields the following corollary.

\begin{corollary} \label{cor_dominated_eigenvector_order_cont_norm}
	Suppose that the complex Banach lattice $E$ has order continuous norm, and let $T$ be a positive operator on $E$. Let $r > 0$, $\theta \in \bbR$ and let $0 < x \in \ker(1-T)$. Then we have
	\begin{align*}
		\dim[E_x \cap \ker(re^{i\theta}-T)] \le \dim[E_x \cap \ker(re^{in\theta}-T)]
	\end{align*}
	for all $n \in \bbZ$. In particular, if $e^{i\theta}$ is an eigenvalue of $T$ with eigenvector $0 \not= z \in E_x$, then $e^{in\theta}$ is also an eigenvalue of $T$ with an eigenvector in $E_x$ for each $n \in \bbZ$.
\end{corollary}
\begin{proof}
	We apply Theorem~\ref{thm_dominated_eigenvector_dual} to the bi-adjoint $T''$. This yields the corollary since $E$ is an ideal in $E''$.  
\end{proof}

Corollary~\ref{cor_dominated_eigenvector_order_cont_norm} is a generalization of \cite[Theorem 3.4]{Scheffold1971}, where the same result (without the dimension estimate) was shown only for peripheral eigenvalues and only under some additional assumptions on $T$.

\section{(WS)-bounded operators} \label{section_ws_bounded_operators}

To obtain cyclicity results on the peripheral (point) spectrum of a positive operator $T$, it is often assumed in the literature that $T$ be \emph{Abel-bounded}, by which we mean that the set
\begin{align*}
	\{(r-r(T))R(r,T): r > r(T)\}
\end{align*}
is bounded in operator norm; see e.g.~\cite[Lemma~V.4.8 and Theorem~V.4.9]{Schaefer1974} for instances where this condition is of relevance. Alternatively, if $r(T) = 1$ we could impose the condition on $T$ that
\begin{align*}
	\liminf_{n \to \infty} ||T^n|| < \infty \text{;}
\end{align*}
this condition is for example used in \cite[Satz~3.6]{Scheffold1971}. In this section we develop a more general notion of boundedness which contains the two aforementioned examples as special cases and which turns out to be well-suited for the spectral analysis of positive operators. \par 
Our approach is based on the notion of \emph{weighting schemes} which will be introduced in the next definition. Let $\overline{\bbD}\subset \bbC$ be the closed unit disk. We call a function $f: \overline{\bbD} \to \bbC$ \emph{analytic} if $f$ has an analytic extension to some open neighbourhood of $\overline{\bbD}$.

\begin{definition} \label{def_properties_and_weighting_scheme}
	Let $f$ be an analytic function on $\overline{\bbD}$ and let $(f_j)$ be a net of analytic functions on $\overline{\bbD}$. Consider the following conditions:
	\begin{description}
		\item[(WS1)] We have $f(1) = 1$. \par 
		\item[(WS2)] We have $f^{(k)}(0) \ge 0$ for each $k \in \bbN_0$. \par 
		\item[(WS3)] For each $z \in \bbC$, $|z| < 1$, we have $f_j(z) \to 0$.
	\end{description}
	The net $(f_j)$ is called a \emph{weighting scheme} if it fulfils the condition (WS3) and if each function $f_j$ fulfils the conditions (WS1) and (WS2).
\end{definition}

Our motivation to call such nets $(f_j)$ in Definition~\ref{def_properties_and_weighting_scheme} \emph{weighting schemes} will become apparent before Definition~\ref{def_ws_bounded} below. The following remark is obvious, but we state it explicitly for later reference.

\begin{remark} \label{rem_subnet_of_weighting_scheme_is_weighting_scheme}
	Every subnet of a weighting scheme is itself a weighting scheme.
\end{remark}

If analytic functions $f_1,...,f_n: \overline{\bbD} \to \bbC$ fulfil the condition (WS1), then of course their product fulfils (WS1) as well. Similarly, if $f_1,...,f_n$ fulfil (WS2), then it follows from the product rule for differentiation that their product also fulfils (WS2). \par 
In fact, we are not merely interested in analytic functions on $\overline{\bbD}$ themselves, but also in the coefficients of their power series expansions around $0$. Hence, we shall now explain how the conditions (WS1)-(WS3) can be expressed by means of those coefficients.

\begin{remark} \label{remark_on_two_properties_of_analytic_functions}
	Let $f$ be an analytic function on $\overline{\bbD}$ and let $f(z) = \sum_{k=0}^\infty a_kz^k$ be the power series expansion of $f$ around $0$. Then the following properties hold true:
	\begin{enumerate}[(a)]
		\item The function $f$ fulfils (WS1) if and only if $\sum_{k=0}^\infty a_k = 1$. \par 
		\item The function $f$ fulfils (WS2) if and only if $a_k \ge 0$ for each $k \in \bbN_0$. \par 
		\item If $f$ fulfils (WS1) and (WS2), then we have $|f(z)| \le 1$ for each $z \in \overline{\bbD}$. If, in addition, $f$ is not constant, then we even have $|f(z)| < 1$ whenever $|z| < 1$.
	\end{enumerate}
\end{remark}
\begin{proof}
	Assertions (a) and (b) are clear; assertion (c) follows from (a) and (b).  
\end{proof}

\begin{proposition} \label{prop_weighting_scheme_conv_of_coefficients}
	Let $(f_j)$ be a net of analytic functions on $\overline{\bbD}$ and suppose that each function $f_j$ fulfils (WS1) and (WS2). Moreover, for each $j$ let $f_j(z) = \sum_{k=0}^\infty a_{j,k}z^k$ be the power series expansion of $f_j$ around $0$. Then the following assertions are equivalent:
	\begin{enumerate}[(i)]
		\item The net $(f_j)$ is a weighting scheme, i.e. $f_j(z) \to 0$ whenever $z \in \bbC$, $|z| < 1$. \par 
		\item For each $k \in \bbN_0$ we have $a_{j,k} \overset{j}{\to} 0$.
	\end{enumerate}
\end{proposition}
\begin{proof}
	``(i) $\Rightarrow$ (ii)'' By Remark~\ref{remark_on_two_properties_of_analytic_functions}(c) each function $f_j$ is bounded by $1$ on $\overline{\bbD}$. Hence, the functions $f_j$ converge to $0$ uniformly on compact subsets of the open unit disk. It then follows from Cauchy's integral formula that $f_j^{(k)}(0) \overset{j}{\to} 0$ for each $k \in \bbN_0$. Since $a_{j,k} = \frac{f_j^{(k)}(0)}{k!}$, this implies (ii). \par 
	``(ii) $\Rightarrow$ (i)'' Let $|z| < 1$, let $\varepsilon > 0$ and choose $k_0 \in \bbN$ such that $\sum_{k=k_0}^{\infty} |z|^k < \frac{\varepsilon}{2}$. For all sufficiently large $j$, say $j \ge j_0$, we have $a_{j,k} \le \frac{\varepsilon}{2k_0}$ for $k=0,...,k_0-1$. Using that all coefficients $a_{j,k}$ are at most $1$, we conclude for all $j \ge j_0$ that
	\begin{align*}
		|f_j(z)| \le \sum_{k=0}^{k_0-1} a_{j,k} + \sum_{k=k_0}^\infty |z|^k \le \varepsilon \text{.}
	\end{align*}
Hence, $f_j(z) \to 0$.  
\end{proof}

One way to obtain examples of weighting schemes is to consider powers of a fixed function:

\begin{proposition} \label{prop_powers_as_weighting_schemes}
	Let $f$ be a non-constant analytic function on $\overline{\bbD}$ which fulfils (WS1) and (WS2). Then the sequence $(f^j)_{j \in \bbN_0}$ is a weighting scheme.
\end{proposition}
\begin{proof}
	Clearly, all powers $f^j$ fulfil the conditions (WS1) and (WS2). It follows from Remark~\ref{remark_on_two_properties_of_analytic_functions}(c) that the sequence $(f^j)_{j \in \bbN_0}$ also fulfils (WS3).  
\end{proof}

Since we have now proved some elementary facts about weighting schemes, it is time to give a few concrete examples of them.

\begin{examples} \label{examp_weighting_schemes}
	The following sequences $(f_j)$ of analytic functions $f_j: \overline{\bbD} \to \bbC$ are weighting schemes:
	\begin{enumerate}[(a)]
		\item $f_j(z) = z^j$, $j \in \bbN_0$. \par 
		\item $f_j(z) = \frac{\lambda_j - 1}{\lambda_j - z}$, where $(\lambda_j)_{j\in \bbN}$ is a sequence in $(1,\infty)$ which converges to $1$. \par 
		\item $f_j(z) = \big( \frac{\lambda - 1}{\lambda - z}\big)^j$, $j \in \bbN_0$, where $\lambda > 1$ is a fixed real number. \par 
		\item $f_j(z) = \frac{1}{j}\sum_{k=0}^{j-1} z^k$, $j \in \bbN$. \par 
		\item $f_j(z) = e^{t_j(z-1)}$, where $(t_j)_{j \in \bbN}$ is a sequence in $[0,\infty)$ which converges to $\infty$.
	\end{enumerate}
\end{examples}
\begin{proof}
	(a) This follows from Proposition~\ref{prop_powers_as_weighting_schemes} since the function $z \mapsto z$ is non-constant and clearly satisfies conditions (WS1) and (WS2). \par 
	(b) Clearly, each function $f_j$ satisfies (WS1). Moreover, we can use the geometric series to compute the power series expansion of $f_j$ around $0$; it then follows from Remark~\ref{remark_on_two_properties_of_analytic_functions}(b) that each $f_j$ also fulfils (WS2). Condition (WS3) is obviously fulfilled. \par 
	(c) The function $z \mapsto \frac{\lambda - 1}{\lambda - z}$ clearly fulfils (WS1); computing its power series expansion around $0$ and using Remark~\ref{remark_on_two_properties_of_analytic_functions}(b) we see that it also fulfils (WS2). Hence the assertion follows from Proposition~\ref{prop_powers_as_weighting_schemes}. \par 
	(d) Clearly, each function $z \mapsto \frac{1}{j}\sum_{k=0}^{j-1} z^k$ fulfils condition (WS1) and, due to part (b) of Remark~\ref{remark_on_two_properties_of_analytic_functions}, also (WS2). Proposition~\ref{prop_weighting_scheme_conv_of_coefficients} now yields that (WS3) is also fulfilled. \par 
	(e) For this sequence, all three conditions (WS1), (WS2) and (WS3) follow from well-known properties of the exponential function.  
\end{proof}

Suppose that $f: \overline{\bbD} \to \bbC$ is an analytic function. If $T$ is a bounded linear operator on a complex Banach space $X$ with spectral radius $r(T) = 1$, then the analytic functional calculus $f(T)$ is well-defined; in fact, we have
\begin{align*}
	f(T) = \sum_{k=0}^\infty a_k T^k
\end{align*}
where $f(z) = \sum_{k=0}^\infty a_kz^k$ is the power series expansion of $f$ around $0$. If $f$ fulfils (WS1) and (WS2), then Remark~\ref{remark_on_two_properties_of_analytic_functions} says that the coefficients $a_k$ are non-negative and that they sum up to $1$; hence, $f(T)$ is a weighted sum of the power $T^k$, $k \in \bbN_0$. If $(f_j)$ is a weighting scheme, then each operator $f_j(T)$ is such a weighted sum and condition (WS3) means that the powers $T^k$ with small exponent $k$ become less important as $j$ increases (see Proposition~\ref{prop_weighting_scheme_conv_of_coefficients}(ii)). The net of operators $(f_j(T))$ which we obtain by applying a weighting scheme $(f_j)$ to an operator $T$ might remind the reader of the notion of so-called Lotz-R\"abiger nets, see e.g.~\cite[Definition~1.2]{Emelyanov2007}. However, the relation between weighting schemes and Lotz-R\"abiger nets does not seem to be clear: For example, let $T$ be power-bounded, $r(T) = 1$. Then $(f_j(T))$ is not a Lotz-R\"abiger net in general if $f_j(z) = z^j$ for $j \in \bbN_0$; however, $(f_j(T))$ is a Lotz-R\"abiger net if $f_j(z) = \frac{1}{j}\sum_{k=0}^{j-1}z^k$ for $j \in \bbN$. We shall not pursue a detailed analysis of this here.

We are now ready to introduce the boundedness notion for linear operators that we mentioned at the beginning of the section.

\begin{definition} \label{def_ws_bounded}
	Let $X$ be a complex Banach space and let $T \in \calL(X)$, $r(T) = 1$. The operator $T$ is called \emph{(WS)-bounded} if there is a weighting scheme $(f_j)_{j \in J}$ such that the set $\{f_j(T): j \in J\}$ is bounded in operator norm.
\end{definition}

A related concept called \emph{P-boundedness} can be found at the end of \cite{Bernau1990}. Note that the notion of (WS)-boundedness of an operator $T$ only makes sense if $r(T) = 1$. Indeed, if $r(T) < 1$, then it follows from Proposition~\ref{prop_weighting_scheme_conv_of_coefficients}(ii) that $f_j(T) \to 0$ for every weighting scheme $(f_j)_{j \in J}$; therefore, $\{f_j(T): j \in J\}$ is bounded whenever the weighting scheme $(f_j)_{j \in J}$ is a sequence. On the other hand, if $r(T) > 1$, then $f_j(T)$ might not even be defined. In order to extend the notion of (WS)-boundedness to general operators $T$ with $r(T) > 0$ we could of course call $T$ (WS)-bounded if $\frac{T}{r(T)}$ is so. However, we prefer to simply state all our results on (WS)-bounded operators under the assumption that $r(T) = 1$. The reader can then easily obtain the general case by a rescaling argument. \par
In Examples~\ref{examp_weighting_schemes} we listed several weighting schemes. We can now use those examples to give several sufficient conditions for an operator to be (WS)-bounded.

\begin{examples} \label{examp_ws_bounded}
	Let $X$ be a complex Banach space and let $T \in \calL(X)$, $r(T) = 1$. Each of the following conditions is sufficient for $T$ to be (WS)-bounded:
	\begin{enumerate}[(a)]
		\item $T$ is power-bounded or, more generally, $\liminf_{j \to \infty} ||T^j|| < \infty$. \par 
		\item $T$ is Abel-bounded or, more generally, $\liminf_{\lambda \downarrow 1} (\lambda-1)||R(\lambda,T)|| < \infty$. \par 
		\item There is a $\lambda > 1$ such that $\liminf_{j \to \infty}||[(\lambda-1)R(\lambda,T)]^j|| < \infty$. \par 
		\item $T$ is Ces\`{a}ro-bounded or, more generally, $\liminf_{j \to \infty} ||\frac{1}{j}\sum_{k=0}^{j-1} T^k|| < \infty$. \par 
		\item We have $\liminf_{t \to \infty} ||e^{t(T-1)}|| < \infty$.
	\end{enumerate}
\end{examples}
\begin{proof}
	This is immediate from Examples~\ref{examp_weighting_schemes} and Remark~\ref{rem_subnet_of_weighting_scheme_is_weighting_scheme}.  
\end{proof}

By Example~\ref{examp_ws_bounded}(a) a power-bounded operator is always (WS)-bounded. In fact, power-boundedness is the most elementary special case of (WS)-boundedness, and one can easily see that if an operator $T \in \calL(X)$ on a complex Banach space $X$ is power-bounded, then the set $\{f_j(T): j \in J\} \subset \calL(X)$ is actually bounded for each weighting scheme $(f_j)_{j \in J}$.

Let us also discuss how (WS)-boundedness is related to other boundedness conditions which appear in Examples~\ref{examp_ws_bounded}: Apparently, the notion of (WS)-boundedness is weaker then other well-known conditions such as Abel-boundedness or Ces\`{a}ro-boundedness. We should point out the each Ces\`{a}ro-bounded operator is automatically Abel-bounded (see \cite[Theorem~1.7]{Emilion1985}) and that for positive operators on Banach lattices, also the converse holds true (see \cite[Paragraph~1.5]{Emilion1985}). Therefore, it is natural to ask whether for positive operators $T$ the boundedness of the set $\{f_j(T):j \in J\}$ for \emph{one} weighting scheme $(f_j)_{j \in J}$ implies that the set is automatically bounded for \emph{every} weighting scheme $(f_j)_{j \in J}$. In fact this is not true even for lattice homomorphisms. In \cite[Section~2]{Derriennic1973} one can find an example of a lattice isomorphism $T$ on an $L^1$-space which is Ces\`{a}ro-bounded, but not power-bounded. Conversely we now give an example of a lattice homomorphism $T$ on an $L^1$-space such that the sequence of its Ces\`{a}ro means has no bounded subsequence, but such that $\liminf_{j \to \infty} ||T^j|| < \infty$.

\begin{example} \label{examp_subsequence_of_powers_and_of_cesaro_sums}
	There is an $L^1$-space $E$ and a lattice homomorphism $T \in \calL(E)$, $r(T) = 1$, such that $\liminf_{j \to \infty} ||T^j|| < \infty$, but $\lim_{j \to \infty} ||\frac{1}{j}\sum_{k=0}^{j-1} T^k|| = \infty$. To construct such an example, we consider for each $m \in \bbN$ the sequence $a^{(m)} = (a_l^{(m)})_{l \in \bbN} \in l^\infty(\bbN;\bbC)$ which is given by
	\begin{align*}
		a_l^{(m)} = 
		\begin{cases}
			2^{\frac{1}{(m-1)!}} \quad & \text{if } 1 \le l < m! \\
			\frac{1}{2^m} \quad & \text{if } l = m! \\
			1 \quad & \text{if } m! < l \text{.}
		\end{cases}
	\end{align*}
	Let $S$ be the right shift on $l^1 := l^1(\bbN;\bbC)$; moreover, for each $m \in \bbN$ we denote by $M_m$ the multiplication operator on $l^1$ with symbol $a^{(m)}$ and we define $T_m := SM_m$. Let $E = l^1(\bbN;l^1) \simeq l^1(\bbN \times \bbN;\bbC)$ and $\calL(E) \ni T := \oplus_{m\in\bbN} T_m$. Note that $T$ indeed maps $E$ to $E$ since $||T_m|| = ||M_m|| \le 2$ for each $m \in \bbN$. Clearly, $T$ is a lattice homomorphism. \par	
	To show that $T$ has the other claimed properties, let us analyse the powers of each operator $T_m$: As somewhat tedious computation shows that $T_m^j = S^j \tilde M_{m,j}$ for every $j \in \bbN_0$, where $\tilde M_{m,j}$ is the multiplication operator whose symbol $\tilde a^{(m,j)} = (\tilde a_l^{(m,j)})_{l \in \bbN} \in l^\infty(\bbN;\bbC)$ is given by
	\begin{align*}
		\tilde a_l^{(m,j)} =
		\begin{cases}
			2^{\frac{j}{(m-1)!}} \quad & \text{if } 1 \le l \le m! - j \\
			2^{\frac{m! - l}{(m-1)!} - m} \quad & \text{if } m! - j < l \le m! \\
			1 \quad & \text{if }  m! < l \text{.}
		\end{cases}
	\end{align*}
	Note that in the second of the above formulas we always have $2^{\frac{m! - l}{(m-1)!} - m} \le 1$. \par 
	In order to prove that $\liminf_{j \to \infty} ||T^j|| < \infty$ we define $j_h := h!$ for all $h \in \bbN$ and we will now show that $||T_m^{j_h}|| \le 2$ for all $m,h \in \bbN$: If $h \ge m$, then the first case in the above formula for $\tilde a_l^{(m,j_h)}$ never occurs and we conclude that $||\tilde M_{m,j_h}||\le 1$. If $h < m$, then the first case in the above formula for $\tilde a_l^{(m,j_h)}$ fulfils the estimate $\tilde a_l^{(m,j_h)} = 2^{\frac{j_h}{(m-1)!}} \le 2^{\frac{(m-1)!}{(m-1)!}} = 2$. Hence, we have $||\tilde M_{m,j_h}|| \le 2$ in this case. To sum up, we obtain $||T_m^{j_h}|| = ||\tilde M_{m,j_h}|| \le 2$ for all $m,h \in \bbN$ and this implies that $||T^{j_h}|| \le 2$ for each $h \in \bbN$. Thus, $\liminf_{j \to \infty} ||T^j|| < \infty$. \par
	It remains to show that the norm of the Ces\`{a}ro sums of $T$ converges to $\infty$. To this end, let $e_1 = (1,0,0,...) \in l^1$. If $m,j \in \bbN$ such that $m! \le j < (m+1)!$, then we have
	\begin{align*}
		& ||\frac{1}{j}\sum_{k=0}^{j-1} T_m^k e_1||_{l^1} \ge ||\frac{1}{(m+1)!}\sum_{k=0}^{m!-1}T_m^ke_1||_{l^1} = \frac{1}{(m+1)!}\sum_{k=0}^{m!-1} 2^{\frac{k}{(m-1)!}} = \\
		& = \frac{2^m - 1}{(m+1)!(2^{\frac{1}{(m-1)!}}-1)} = \frac{2^m - 1}{(m+1)m} \cdot \frac{1}{(m-1)! (2^{\frac{1}{(m-1)!}} - 1)} =: c(m) \text{.}
	\end{align*}
	From the well-known convergence result $\lim_{n \to \infty} (1+\frac{1}{n})^n = e > 2$ one easily derives that $c(m) \ge \frac{2^m - 1}{(m+1)m}$ for all sufficiently large $m$. Hence, $c(m) \to \infty$ as $m \to \infty$. \par 
	Now, for any constant $C > 0$ we can find a number $m_0\in \bbN$ such that $c(m) \ge C$ for all $m \ge m_0$. If $j \ge m_0!$ we may choose $m \ge m_0$ such that $m! \le j < (m+1)!$ and thus we obtain
	\begin{align*}
		||\frac{1}{j}\sum_{k=0}^{j-1} T^k|| \ge ||\frac{1}{j}\sum_{k=0}^{j-1} T_m^k|| \ge c(m) \ge C \text{.}
	\end{align*}
	So we have indeed $\lim_{j \to \infty} ||\frac{1}{j}\sum_{k=0}^{j-1}T^k|| = \infty$.
\end{example}

Example~\ref{examp_subsequence_of_powers_and_of_cesaro_sums} and the above mentioned example in \cite[Section~2]{Derriennic1973} show that (WS)-bound\-ed\-ness of a positive operator $T$ does not imply that $\{f_j(T): j  \in J\}$ is bounded for each weighting scheme $(f_j)_{j \in J}$. However, the situation is different for compact positive operators; this follows from Corollary~\ref{cor_pos_only_poles_and_spec_circle_ws_bounded} below which in turn is a consequence of the next proposition. Recall that a spectral value $\lambda_0$ of an operator $T$ on a complex Banach space $X$ is called an \emph{$m$-th order pole of the resolvent} (where $m \in \bbN$) if $\lambda_0$ is an isolated point in the spectrum $\sigma(T)$ and if the analytic mapping
\begin{align*}
	\bbC \setminus \sigma(T) \to \calL(X) \text{,} \quad \lambda \mapsto R(\lambda,T)
\end{align*}
has a pole of order $m$ at $\lambda_0$. 

\begin{proposition} \label{prop_spec_rad_pol_of_resolvent_and_ws_bounded}
	Let $X$ be a complex Banach space, $T \in \calL(X)$ and suppose that $r(T) = 1$ is a spectral value of $T$ and an $m$-th order pole of the resolvent, $m \in \bbN$. Then $T$ is (WS)-bounded if and only if $m = 1$. 
\end{proposition}

For the proof we need the following lemma which will also be useful later on. 

\begin{lemma} \label{lem_weighting_scheme_monotone_sequence}
	Let $(f_j)_{j \in J}$ be a weighting scheme and let $f_j(z) = \sum_{k=0}^\infty a_{j,k}z^k$ be the power series expansion of each $f_j$ around $0$. Moreover, let $(r_k)_{k \in \bbN_0} \subset [0,\infty)$ be a non-decreasing sequence of real numbers and assume that $\limsup_{k\to \infty} \sqrt[k]{r_k} \le 1$ such that the series $\sum_{k=0}^\infty a_{j,k} r_k$ converges for each $j \in J$. \par 
	If the set $\{\sum_{k=0}^\infty a_{j,k}r_k: \, j \in J\}$ is bounded in $\bbR$, then the sequence $(r_k)_{k \in \bbN_0}$ must be bounded.
\end{lemma}
\begin{proof}
	Assume that $(r_k)_{k \in \bbN_0}$ is unbounded and let $C > 0$. Then we have $r_k \ge C+1$ for all sufficiently large $k$, say $k \ge k_0$. It follows from Proposition~\ref{prop_weighting_scheme_conv_of_coefficients} that for some $j \in J$ we have $a_{j,k} \le \frac{1}{k_0(C+1)}$ for all $k = 0,...,k_0-1$. Using that $\sum_{k=0}^\infty a_{j,k} = 1$ and that $a_{j,k} \ge 0$ (see Remark~\ref{remark_on_two_properties_of_analytic_functions}(a) and (b)) we obtain
	\begin{align*}
		\sum_{k=0}^\infty a_{j,k} r_k & \ge (C+1) \sum_{k=k_0}^\infty a_{j,k} = (C+1)\big(1 - \sum_{k=0}^{k_0-1}a_{j,k}\big) \ge C \text{.}
	\end{align*}
	Thus, the set $\{\sum_{k=0}^\infty a_{j,k}r_k: \, j \in J\}$ is unbounded.  
\end{proof}

\begin{proof}[Proof of Proposition~\ref{prop_spec_rad_pol_of_resolvent_and_ws_bounded}]
	``$\Leftarrow$'' If $m=1$, then $T$ is Abel-bounded and in particular (WS)-bounded. \par 
	``$\Rightarrow$'' Assume for a contradiction that $m \ge 2$. Then there is a vector $y \in \ker((1-T)^2) \setminus \ker(1-T)$. If $x := (T-1)y$, then $x$ and $y$ are linearly independent and their span is $T$-invariant. Using the representation matrix of $T$ on the span of $x$ and $y$ with respect to the basis $(x,y)$, a short computation shows that $T^k y = kx + y$ for each $k \in \bbN_0$. \par 
	Choose a functional $x' \in X'$ such that $\langle x', x \rangle > 0$. Since $T$ is (WS)-bounded, there is a weighting scheme $(f_j)_{j \in J}$ such that the operator family $\{f_j(T): j \in J\}$ is bounded in operator norm. Denoting by $f_j(z) = \sum_{k=0}^\infty a_{j,k}z^k$ the power series expansion of each $f_j$ around $0$, we conclude that the set
	\begin{align*}
		\{\sum_{k=0}^\infty a_{j,k}\langle x',T^ky \rangle: j \in J\}
	\end{align*}
	is bounded in $\bbC$. Moreover, we have $\sum_{k=0}^\infty a_{j,k} \langle x', y\rangle = \langle x', y \rangle$ for each $j \in J$, so the set
	\begin{align*}
		\{\sum_{k=0}^\infty a_{j,k} \langle x',kx \rangle: j \in J \} = \{ \sum_{k=0}^\infty a_{j,k} \langle x', T^ky \rangle - \sum_{k=0}^\infty a_{j,k} \langle x', y \rangle: j \in J \}
	\end{align*}
	must be bounded in $\bbC$ as well. By Lemma~\ref{lem_weighting_scheme_monotone_sequence} this implies that the sequence $(\langle x',kx\rangle )_{k \in \bbN_0}$ is bounded, which is a contradiction.  
\end{proof}

\begin{corollary} \label{cor_pos_only_poles_and_spec_circle_ws_bounded}
	Let $T$ be a positive operator on a complex Banach lattice $E$, $r(T) = 1$. If $\sigma_{\per}(T)$ consists of poles of the resolvent, then the following assertions are equivalent:
	\begin{enumerate}[(i)]
		\item For each weighting scheme $(f_j)_{j \in J}$, the set $\{f_j(T): j \in J\}$ is bounded in operator norm. \par 
		\item $T$ is (WS)-bounded, i.e.~there is at least one weighting scheme $(f_j)_{j \in J}$ such that the set $\{f_j(T): j \in J\}$ is bounded in operator norm. \par 
		\item The spectral radius $r(T) = 1$ is a first order pole of the resolvent. \par 
		\item All peripheral spectral values of $T$ are first order poles of the resolvent.
	\end{enumerate}
\end{corollary}
\begin{proof}
	The implication ``(i) $\Rightarrow$ (ii)'' is obvious. To see ``(ii) $\Rightarrow$ (iii)'', note that $r(T) \in \sigma(T)$ since $T$ is a positive operator (see \cite[Proposition~V.4.1]{Schaefer1974}) and apply Proposition~\ref{prop_spec_rad_pol_of_resolvent_and_ws_bounded}. \par 
	The implication ``(iii) $\Rightarrow$ (iv)'' is a consequence of the estimate $|R(\lambda,T)x| \le R(|\lambda|,T)|x|$ which holds true for all $x \in E$ and for all $\lambda$ with $|\lambda| > 1$ since $T$ is positive. To prove ``(iv) $\Rightarrow$ (i)'', note that (iv) implies that $T$ is power-bounded, which in turn implies (i).  
\end{proof}

According to Corollary~\ref{cor_pos_only_poles_and_spec_circle_ws_bounded} the notion of (WS)-boundedness does not yield anything new for positive compact operators; in fact it follows from the equivalence of conditions (i) and (ii) in the corollary that such an operator is (WS)-bounded if and only if it is power-bounded.

Our last result in this section is a lemma on the behaviour of the powers $T^n$ of a (WS)-bounded positive operator $T$ when applied to certain positive vectors. This lemma will be the key to our spectral results in the subsequent sections.

\begin{lemma} \label{lem_ws_bounded_op_monotone_vector}
	Let $T$ be a positive operator on a complex Banach lattice $E$, $r(T) = 1$. Suppose that $0 \le x \in E$ and that the sequence $(T^nx)_{n \in \bbN_0}$ is non-decreasing. If $T$ is (WS)-bounded, then $(T^nx)_{n \in \bbN_0}$ is bounded in norm.
\end{lemma}
\begin{proof}
	Let $0 \le x' \in E'$. Since $T$ is (WS)-bounded, there is a weighting scheme $(f_j)_{j \in J}$ such that the operator family $\{f_j(T): j \in J\}$ is bounded in operator norm. If $f_j(z) = \sum_{k=0}^\infty a_{j,k}z^k$ denotes the power series expansion of each $f_j$ around $0$, then the set
	\begin{align*}
		\{\sum_{k=0}^\infty a_{j,k}\langle x',T^kx \rangle: j \in J\}
	\end{align*}
	is bounded in $\bbC$. Lemma~\ref{lem_weighting_scheme_monotone_sequence} now implies that the sequence $(\langle x', T^kx\rangle)_{k \in \bbN_0}$ is bounded. Due to the Uniform Boundedness Principle this yields the assertion.  
\end{proof}

\section{The peripheral point spectrum of positive operators} \label{section_per_pnt_spec_general}

Now we can apply our results from the previous sections to prove several sufficient conditions for the peripheral point spectrum of a positive operator to be cyclic. For several known results on this topic we refer the reader for example to \cite[Sections~V.4 and V.5]{Schaefer1974}, \cite[Sections~4 and 5]{Grobler1995} and \cite[Sections~2.2, 2.3 and 3.2]{Krieger1969}. We point out that most of the results in these references are somewhat different in nature from ours below. Some of our theorems were inspired by (and thus have some similarity with) results in \cite[Section~3]{Scheffold1971}. Our first theorem is concerned with adjoint operators.

\begin{theorem} \label{thm_cyclcic_per_pnt_spec_ws_bounded_and_dual}
	Let $T$ be a positive operator on a complex Banach lattice $E$, $r(T) = 1$. Suppose that $E$ has a pre-dual Banach lattice and that $T$ has a pre-adjoint. If $T$ is (WS)-bounded, then we have $\dim \ker(e^{i\theta}-T) \le \dim \ker(e^{in\theta}-T)$ for each $n \in \bbZ$ and each $\theta \in \bbR$. In particular, the peripheral point spectrum of $T$ is cyclic.
\end{theorem}
\begin{proof}
	Let $\theta \in \bbR$. We may assume that $e^{i\theta}$ is an eigenvalue of $T$ because otherwise the assertion is trivial. If $0 \not= z \in E$ is a corresponding eigenvector, then we have $|z| \le T|z|$ and iterating this inequality, we obtain that the sequence $(T^n|z|)_{n \in \bbN_0}$ is non-decreasing. Lemma~\ref{lem_ws_bounded_op_monotone_vector} thus implies that the sequence is bounded in norm and hence, it converges with respect to the weak${}^*$-topology to an element $x \ge |z|$. Since $T$ is continuous with respect to the weak${}^*$-topology, $x$ is a fixed point of $T$. \par 
	Now, let $\dim \ker(e^{i\theta}-T) \ge m \in \bbN$, and choose $m$ vectors $z_1,...,z_m \in \ker(e^{i\theta}-T)$ which are linearly independent. As shown above, we can find vectors $x_1,...,x_m \in \ker(1-T)$ which dominate the vectors $|z_1|,...,|z_m|$, respectively. Now we define $x = x_1 + ... + x_m$ and Theorem~\ref{thm_dominated_eigenvector_dual} then implies that
	\begin{align*}
		m \le \dim[E_x \cap \ker(e^{i\theta}-T)] \le \dim[E_x \cap \ker(e^{in\theta}-T)] \le \dim \ker(e^{in\theta}-T)
	\end{align*}
	for each $n \in \bbZ$. This yields the assertion.  
\end{proof}

A related result, however with somewhat stronger assumptions and without a dimension estimate, can be be found in \cite[Satz~3.3]{Scheffold1971}. \par 

In \cite[Theorem~3.5]{Scheffold1971} Scheffold proved that on AL-spaces and on reflexive Banach lattices each Abel-bounded positive operator has cyclic peripheral point spectrum. In fact, a short inspection of the proof shows that this result even holds true in the larger class of KB-spaces (see \cite[Definition~2.4.11]{Meyer-Nieberg1991} for a definition of the notion \emph{KB-space}). By combining Scheffold's proof with our results from Sections~\ref{section_eigenvalues_with_dom_eigenvectors} and \ref{section_ws_bounded_operators} we obtain the following generalization of Scheffold's theorem:

\begin{theorem} \label{thm_cyclic_per_pnt_spec_ws_bounded_on_kb_space}
	Let $T$ be a positive operator on a complex Banach lattice $E$, $r(T) = 1$. If $E$ is a KB-space and $T$ is (WS)-bounded, then we have $\dim \ker(e^{i\theta}-T) \le \dim \ker(e^{in\theta}-T)$ for each $n \in \bbZ$ and each $\theta \in \bbR$. In particular, the peripheral point spectrum of $T$ is cyclic.
\end{theorem}
\begin{proof}
	Let $\theta \in \bbR$ such that $e^{i\theta}$ is an eigenvalue of $T$. If $0\not=z \in E$ is a corresponding eigenvector, then the sequence $(T^n|z|)_{n \in \bbN_0}$ is non-decreasing; since $T$ is (WS)-bounded, the sequence is therefore bounded in norm due to Lemma~\ref{lem_ws_bounded_op_monotone_vector}. As $E$ is a KB-space, the limit $x := \lim_{n \to \infty}T^n|z|$ exists in norm and clearly $x \in \ker(1-T)$. \par 
	Now, let $\dim \ker(e^{i\theta} -T) \ge m \in \bbN$ and choose $m$ linearly independent vectors $z_1,...,z_m \in \ker(e^{i\theta}-T)$. As we have seen above, we can find corresponding vectors $x_1,...,x_m \in \ker(1-T)$ which dominate the vectors $|z_1|,...,|z_m|$, respectively, and we set $x := x_1 + ... + x_m$. Since every KB-space has order continuous norm (this follows from \cite[Theorem~2.4.2\,(i),\,(iii)]{Meyer-Nieberg1991}), we can apply Corollay~\ref{cor_dominated_eigenvector_order_cont_norm} which yields
	\begin{align*}
		m \le \dim[E_x \cap \ker(e^{i\theta}-T)] \le \dim[E_x \cap \ker(e^{in\theta}-T)] \le \dim \ker(e^{in\theta}-T)
	\end{align*}
	for each $n \in \bbZ$.  
\end{proof}

One should point out that a positive operator $T$ on a KB-space $E$ with $r(T) = 1$ does not need to have cyclic peripheral point spectrum if we do not impose any boundedness condition on $T$; for a counterexample we refer to \cite[Example~C-III.4.4]{Arendt1986} (in fact, this reference contains an example of a $C_0$-semigroup $(e^{tA})_{t \ge 0}$ rather then of an operator $T$, but the reader can obtain the desired single operator example by defining $T = e^{t_0A}$ for some $t_0 \not\in 2\pi\bbQ$).

Recall that an operator $T$ on a Banach space $X$ is called \emph{mean ergodic} if the sequence of Ces\`{a}ro means $\frac{1}{n}\sum_{k=0}^{n-1} T^k$ strongly converges to an operator $P \in \calL(X)$. By the Uniform Boundedness Principle, every mean ergodic operator is Ces\`{a}ro-bounded and thus (WS)-bounded. It is interesting to note that for mean ergodic operators with spectral radius $1$ the peripheral point spectrum is cyclic not only on KB-spaces but even on Banach lattices with order continuous norm. This result is essentially known from \cite[Theorem~3.5]{Scheffold1971} (in fact it was stated there only on a smaller class of Banach lattices, but one can easily see that the proof also works on Banach lattices with order continuous norm). The following theorem is a generalization of this result since it also contains an estimate on the dimension of the corresponding eigenspaces.

\begin{theorem} \label{thm_cyclic_per_pnt_spec_ergodic_order_cont_norm}
	Let $T$ be a positive operator on a complex Banach lattice $E$, $r(T) = 1$. If $E$ has order-continuous norm and $T$ is mean ergodic, then $\dim \ker(e^{i\theta}-T) \le \dim \ker(e^{in\theta}-T)$ for each $n \in \bbZ$ and each $\theta \in \bbR$. In particular, the peripheral point spectrum of $T$ is cyclic.
\end{theorem}
\begin{proof}
	Let $e^{i\theta}$ ($\theta \in \bbR$) be an eigenvalue of $T$ with corresponding eigenvector $0 \not= z \in E$ and let $P$ be the mean ergodic projection of $T$. Since $P$ is the strong limit of the sequence $(\frac{1}{n}\sum_{k=0}^{n-1}T^k)_{n \in \bbN_0}$ and since $T^k|z| \ge |z|$ for each $k \in \bbN_0$, we conclude that $|z| \le P|z| \in \ker(1-T)$. \par 
	Now, let $\dim \ker(e^{i\theta}-T) \ge m \in \bbN$ and choose $m$ vectors $z_1,...,z_m \in \ker(e^{i\theta}-T)$ which are linearly independent. If we define $x := P|z_1| + ... + P|z_m|$, then $x \in \ker(1-T)$ and $z_1,...,z_n \in E_x$, so we can apply Corollary~\ref{cor_dominated_eigenvector_order_cont_norm} to conclude that
	\begin{align*}
		m \le \dim[E_x \cap \ker(e^{i\theta}-T)] \le \dim[E_x \cap \ker(e^{in\theta}-T)] \le \dim \ker(e^{in\theta}-T)
	\end{align*}
	for every $n \in \bbZ$. This proves the Theorem.  
\end{proof}

If $E$ does not have order-continuous norm but is still order complete, we can prove for mean ergodic operators that at least a certain part of $\sigma_{\per,\pnt}(T)$ is still cyclic, see Proposition~\ref{prop_ergodic_op_on_order_complete_space_rational_per_pnt_spec} below.

One might wonder if we can drop the condition on $T$ to be mean ergodic in Theorem~\ref{thm_cyclic_per_pnt_spec_ergodic_order_cont_norm} and replace it by (WS)-boundedness or some stronger boundedness condition, say power-boundedness. The following example shows that the answer is negative.

\begin{example} \label{examp_non_cyclic_per_point_spec_on_am_space_without_unit}
	Let $\{1\} \not= G$ be a closed subgroup of the complex unit circle $\bbT$. Then there is an AM-space $E$ (without unit) and a contractive, positive operator $T \in \calL(E)$ such that $\sigma_{\per,\pnt}(T) = G \setminus \{1\}$. If $G$ is finite, then $E$ can be chosen to have order continuous norm. \par 
	To construct such an example, let $\hat G$ be the dual group of $G$, i.e.~$\hat G\simeq G$ if $G$ is finite and $\hat G \simeq \bbZ$ if $G = \bbT$. Since $\hat G$ is a cyclic group, it contains a generating element $\sigma_0$, and we define $S$ to be the shift operator on $l^\infty(\hat G; \bbC)$ which is given by $(Sf)(\sigma) = f(\sigma_0\sigma)$. We clearly have $\sigma_{\pnt}(S) = G$. \par 
	Now let $\tilde E := l^\infty(\hat G; \bbC) \times l^\infty(\bbN;\bbC)$ and $E := l^\infty(\hat G; \bbC) \times c_0(\bbN;\bbC) \subset \tilde E$. Note that if $G$ is finite, then $E$ is isometrically lattice isomorphic to $c_0(\bbN;\bbC)$ and thus has order continuous norm (see~\cite[Example~6 on p.\,92]{Schaefer1974}). We define a positive operator $\tilde T \in \calL(\tilde E)$ by $\tilde T(f,g) = (f',g')$ where
	\begin{align*}
		 f' = Sf \qquad \text{and} \qquad g'_n = \frac{n}{n+1}g_{n+1} + \frac{1}{n+1} f(\sigma_0) \quad \text{for all } n \in \bbN \text{.}
	\end{align*}
	Obviously, $\tilde T$ is a Markov operator on $\tilde E$; moreover, $\tilde T$ leaves $E$ invariant, and we define $\calL(E) \ni T := \tilde T|_E$. Now, we claim the following properties of the operators $\tilde T$ and $T$:
	\begin{enumerate}[(a)]
		\item For $\lambda \in \bbT$ and $(f,g) \in \tilde E$ we have $\tilde T(f,g) = \lambda(f,g)$ if and only if $Sf = \lambda f$ and
			\begin{align*}
				(*) \qquad g_n = \lambda^n n\big(\overline{\lambda} g_1 - f(\sigma_0) \sum_{k=2}^n \frac{\overline{\lambda}^k}{k(k-1)}\big) \quad \text{for all } n\ge 2 \text{.}
			\end{align*}
			In this case, we moreover have $\overline{\lambda} g_1 = f(\sigma_0) \sum_{k=2}^\infty \frac{\overline{\lambda}^k}{k(k-1)}$.
		\item We have $||\tilde T|| = r(\tilde T) = 1$ and $\sigma_{\pnt,\per}(\tilde T) = G$. \par 
		\item We have $||T|| = r(T) = 1$ and $\sigma_{\pnt,\per}(T) = G \setminus \{1\}$.
	\end{enumerate}
	\begin{proof}
		(a) Clearly, we have $\tilde T(f,g) = \lambda(f,g)$ if and only if $f \in \ker(\lambda - S)$ and $\lambda g_n = \frac{n}{n+1}g_{n+1} + \frac{1}{n+1} f(\sigma_0)$ for each $n \in \bbN$. A short computation shows that the last condition is equivalent to $(*)$. Moreover, if $(*)$ holds true, then the fact that the sequence $g$ is bounded implies that $\overline{\lambda} g_1 = f(\sigma_0) \sum_{k=2}^\infty \frac{\overline{\lambda}^k}{k(k-1)}$. \par 
		(b) Since $\tilde T$ is a Markov operator, we clearly have $||\tilde T|| = r(\tilde T) = 1$. Now, let $\lambda \in \sigma_{\pnt,\per}(\tilde T)$ and let $0 \not= (f,g) \in \tilde E$ be a corresponding eigenvector. Then $|\lambda| = 1$, and the conditions from (a) are fulfilled. This yields that $f \not= 0$; indeed, if we assumed $f = 0$, then $g_1 = 0$ and thus $g = 0$ due to $(*)$. Hence,  we indeed have $f \not= 0$, which implies that $\lambda \in \sigma_{\pnt}(S) = G$. \par 
		On the other hand, let $\lambda \in G$. Then we can find $0 \not= f \in \ker(\lambda - S)$. Now, let $g_1$ be defined by the equation $\overline{\lambda} g_1 = f(\sigma_0) \sum_{k=2}^\infty \frac{\overline{\lambda}^k}{k(k-1)}$, and define $g$ by $(*)$. Note that we have
		\begin{align*}
			|g_n| = n \, |f(\sigma_0)| \; |\sum_{k=n+1}^\infty \frac{\overline{\lambda}^k}{k(k-1)}| \le |f(\sigma_0)| \; n \sum_{k=n+1}^\infty \frac{1}{k(k-1)} = |f(\sigma_0)|
		\end{align*}
		for each $n \ge 2$. Hence, we indeed have $g \in l^\infty(\bbN;\bbC)$, so $(f,g) \in \tilde E$ is an eigenvector of $\tilde T$ for the eigenvalue $\lambda$. Thus, $\lambda \in \sigma_{\per,\pnt}(\tilde T)$. \par 
		(c) Clearly, $\sigma_{\pnt}(T) \cap \bbT \subset \sigma_{\per,\pnt}(\tilde T) = G$. It easily follows from (a) that $\ker(1 - \tilde T)$ is spanned by $(\mathbbm{1}_G, \mathbbm{1}_\bbN) \not \in E$. Hence, $1 \not \in \sigma_{\pnt}(T)$. On the other hand, let $\lambda \in G \setminus\{1\}$ and let $0 \not= (f,g) \in \tilde E$ be an eigenvector of $\tilde T$ for the eigenvalue $\lambda$. Then $(*)$ is fulfilled, and we have $\overline{\lambda} g_1 = f(\sigma_0) \sum_{k=2}^\infty \frac{\overline{\lambda}^k}{k(k-1)}$ due to (a). Note that the sequence $(\frac{1}{k(k-1)})_{k \ge 2}$ decreases to $0$, and that the sequence of partial sums $(\sum_{k=0}^n \overline{\lambda}^k)_{n \in \bbN_0}$ is bounded since $\lambda \in \bbT \setminus \{1\}$. Hence, it follows from the error estimate in Dirichlet's series convergence test (see \cite[Theorem~6.55]{Giaquinta2004}) that the difference $|\overline{\lambda} g_1 - f(\sigma_0) \sum_{k=2}^n \frac{\overline{\lambda}^k}{k(k-1)}|$ decreases to zero at least with the same rate as $\frac{1}{n(n-1)}$. Equation $(*)$ thus implies that $g \in c_0(\bbN;\bbC)$. Hence, we indeed have $(f,g) \in E$ and thus $\lambda \in \sigma_{\pnt}(T)$. Therefore, $\sigma_{\pnt}(T) \cap \bbT = G \setminus \{1\} \not= \emptyset$. \par 
		We finally conclude that $1 \le r(T) \le ||T|| \le ||\tilde T|| = 1$ which proves the assertions of (c).  
	\end{proof}
\end{example}

A further example of a contractive positive operator with spectral radius $1$ which is defined on an AM-space with order continuous norm, but has non-cyclic peripheral point spectrum, is briefly discussed in Remark~\ref{rem_adaption_of_markov_op_example} in the next section. Another example which shows that the assertion of Theorem~\ref{thm_cyclic_per_pnt_spec_ergodic_order_cont_norm} fails for power-bounded operators which are not mean ergodic can be obtained by modifying a $C_0$-semigroup example from \cite[Example~B-III.2.13]{Arendt1986}. Using the idea of this example it is easy to construct a power-bounded (but non-contractive) operator $T$ on $c_0(\bbN;\bbC)$ with spectral radius $r(T) = 1$ such that $\sigma_{\per,\pnt}(T) = G \setminus \{1\}$ for any given finite subgroup $G \not= \{1\}$ of $\bbT$.

To state our next theorem, recall that a bounded linear operator $T$ on a Banach space $X$ is called \emph{weakly almost periodic} if the set $\{T^n: n \in \bbN_0\} \subset \calL(X)$ is relatively compact with respect to the weak operator topology; equivalently, the set $\{T^nx: n \in \bbN_0\} \subset X$ is relatively compact in the weak topology for each $x \in X$.

\begin{theorem} \label{thm_cyclic_per_pnt_spec_almost_weakly_periodic}
	Let $T$ be a positive operator on a complex Banach lattice $E$, $r(T) = 1$. If $T$ is weakly almost periodic, then we have $\dim \ker(e^{i\theta}-T) \le \dim \ker(e^{in\theta}-T)$ for each $n \in \bbZ$ and each $\theta \in \bbR$. In particular, the peripheral point spectrum of $T$ is cyclic.
\end{theorem}
\begin{proof}
	Let $\calS$ be the closure of $\{T^n: \, n \in \bbN_0\}$ in $\calL(E)$ with respect to the weak operator topology. Then $\calS$ is an abelian compact semi-topological semigroup. Denote by
	\begin{align*}
		\calK := \bigcap_{A \in \calS} A\calS
	\end{align*}
	the so-called \emph{Sushkevich kernel} of $\calS$. Then it can be shown that $\calK$ is an ideal in the semigroup $\calS$ and that $\calK$ is even a group (see \cite[Theorem 4.1 on p.\,104]{Krengel1985}). If $P$ denotes the neutral element in $\calK$, then $P$ is clearly a projection on $E$ and this projection is positive since $P \in \calS$; moreover, it can be shown that the range $\operatorname{rg}P$ of $P$ coincides with the closed linear span of all eigenvectors of $T$ belonging to unimodular eigenvalues (see \cite[Theorem~4.4 on p.\,105 and Theorem~4.5 on p.\,106]{Krengel1985}). Since $T$ commutes with $P$, it leaves its range $\operatorname{rg} P$ invariant. \par 
	Since $\calK$ is an ideal in the semigroup $\calS$ we have $TP \in \calK$ and thus, we can find an element $R \in \calK$ such that $TPR = RTP = P$. The operator $R$ is positive since it is contained in $\calK$, and for the same reason, it commutes with $P$. Hence, $R$ leaves $\operatorname{rg} P$ invariant, and we have $T|_{\operatorname{rg} P} R|_{\operatorname{rg} P} = R|_{\operatorname{rg} P} T|_{\operatorname{rg} P} = P|_{\operatorname{rg} P} = \operatorname{id}_{\operatorname{rg} P}$. Since $\operatorname{rg} P$ is the range of a positive projection, it is a Banach lattice with respect to some new norm (cf. \cite[Proposition III.11.5]{Schaefer1974}), hence $T|_{\operatorname{rg} P}$ and $R|_{\operatorname{rg} P}$ are positive, mutually inverse operators on the complex Banach lattice $\operatorname{rg} P$. Thus, $T|_{\operatorname{rg} P}$ is a lattice isomorphism, and we conclude from Proposition~\ref{prop_eigenvalues_and_funtions_of_lattice_hom}(b) that
	\begin{align*}
		\dim \ker(e^{i\theta}-T) & = \dim \ker(e^{i\theta}-T|_{\operatorname{rg} P}) \le \\
		& \le \dim \ker(e^{in\theta}-T|_{\operatorname{rg} P}) = \dim \ker(e^{in\theta}-T)
	\end{align*}
	for each $n \in \bbZ$ and each $\theta \in \bbR$.  
\end{proof}

A similar approach as in the above proof was used for the spectral analysis of positive $C_0$-semigroups in \cite{Keicher2008}. \par 

In the remainder of this section, we do not analyse the entire peripheral point spectrum of a positive operator, but the part of it which consists (up to rescaling by $r(T)$) of roots of unity.

\begin{definition} \label{def_rational_per_point_spec}
	Let $T$ be an operator on a complex Banach space $X$. Then we call the set
	\begin{align*}
		\{r(T) \cdot e^{i\theta}: \, \theta \in 2\pi \bbQ\} \cap \sigma_{\operatorname{pnt}}(T)
	\end{align*}
	the \emph{rational peripheral point spectrum} of $T$.
\end{definition}

The following result is essentially a single operator version of \cite[C-III.4.3(b)]{Arendt1986} (although the latter reference does not contain a dimension estimate).

\begin{proposition} \label{prop_all_powers_ergodic_implies_that_rational_per_point_spec_is_cyclic}
	Le $T$ be a positive operator on a complex Banach lattice $E$, $r(T) = 1$. If each power $T^N$ is mean ergodic, then we have $\dim \ker(e^{i\theta}-T) \le \dim \ker(e^{in\theta}-T)$ for each $\theta \in 2\pi\bbQ$ and each $n \in \bbZ$. In particular, the rational peripheral point spectrum of $T$ is cyclic.
\end{proposition}
\begin{proof}
	Let $\theta \in 2\pi \bbQ$ and suppose that $e^{i\theta}$ is an eigenvalue of $T$. We have $e^{iN\theta} = 1$ for some $N \in \bbN$ and thus, the eigenspaces $\ker(e^{i\theta}-T)$ and $\ker(e^{in\theta}-T)$ ($n \in \bbZ$) are contained in the fixed space $F:=\ker(1-T^N)$. Since $T^N$ is mean ergodic, it admits a mean ergodic projection $P$; this projection is clearly positive and has $F$ as its range. Therefore, $F$ is a complex Banach lattice with respect to an appropriate norm (this follows from \cite[Proposition III.11.5]{Schaefer1974}). \par 
	Moreover, the operator $T$ leaves $F$ invariant and its restriction $T|_F$ is periodic with period $N$. Since $(T|_F)^{-1} = (T|_F)^{N-1}$ is positive, too, $T|_F$ is a lattice isomorphism on the complex Banach lattice $F$. Therefore we conclude from Proposition~\ref{prop_eigenvalues_and_funtions_of_lattice_hom}(b) that
	\begin{align*}
		\dim \ker(e^{i\theta}-T) = \dim \ker(e^{i\theta}-T|_F) \le \dim \ker(e^{in\theta}-T|_F) = \dim \ker(e^{in\theta}-T)
	\end{align*}
	for each $n \in \bbZ$.  
\end{proof}

It would be interesting to know whether in the situation of Proposition~\ref{prop_all_powers_ergodic_implies_that_rational_per_point_spec_is_cyclic} the entire peripheral point spectrum is cyclic. \par 

Concerning Proposition~\ref{prop_all_powers_ergodic_implies_that_rational_per_point_spec_is_cyclic} one might ask under which conditions one can ensure that all powers of $T$ are mean ergodic. We point out that on Banach lattices with order continuous norm, all powers of $T$ are mean ergodic if $T$ itself is mean ergodic (see \cite[Theorem~2.1.14 and the comment after Theorem~2.1.5]{Emelyanov2007a}). However, on those spaces Theorem~\ref{thm_cyclic_per_pnt_spec_ergodic_order_cont_norm} yields much stronger results then Proposition~\ref{prop_all_powers_ergodic_implies_that_rational_per_point_spec_is_cyclic} anyway. \par 

On order complete Banach lattices the conclusion of Proposition~\ref{prop_all_powers_ergodic_implies_that_rational_per_point_spec_is_cyclic} still holds if we only assume the operator $T$ itself to be mean ergodic. We show this in the following proposition; however, to do so, we need to use a result from the subsequent section.

\begin{proposition} \label{prop_ergodic_op_on_order_complete_space_rational_per_pnt_spec}
	Le $T$ be a positive operator on a complex Banach lattice $E$, $r(T) = 1$. If $T$ is mean ergodic and $E$ is order complete, then we have $\dim \ker(e^{i\theta}-T) \le \dim \ker(e^{in\theta}-T)$ for each $\theta \in 2\pi\bbQ$ and each $n \in \bbZ$. In particular, the rational peripheral point spectrum of $T$ is cyclic.
\end{proposition}
\begin{proof}
	Let $e^{i\theta}$ ($\theta \in 2\pi \bbQ$) be an eigenvalue of $T$ and let $0 \not= z \in E$ be a corresponding eigenvector. If $P$ denotes the mean ergodic projection of $T$, then we have $P|z| \in \ker(1-T)$ and $P|z| = \lim_{n \to \infty} \frac{1}{n} \sum_{k=0}^{n-1} T^k|z| \ge |z|$. \par 
	Now let $\dim \ker(e^{i\theta}-T) \ge m \in \bbN$ and choose $m$ linearly independent vectors $z_1,...,z_m \in \ker(e^{i\theta}-T)$. We define $x = P|z_1| + ... + P|z_m|$ and as shown above, $x$ is an element of the fixed space  $\ker(1-T)$ and dominates all vectors $|z_1|,...,|z_m|$. If endowed with an appropriate norm the principal ideal $E_x$ is an AM-space with unit $x$ (see \cite[the corollary of Proposition~II.7.2]{Schaefer1974}) and the operator $T|_{E_x}$ is a Markov operator on $E_x$. Since $E$ is order complete, so is $E_x$ and we can thus conclude from Proposition~\ref{prop_markov_op_on_order_complete_space_rational_per_pnt_spec} that
	\begin{align*}
		m \le \dim\ker(e^{i\theta}-T|_{E_x}) \le \dim\ker(e^{in\theta}-T|_{E_x}) \le \dim \ker(e^{in\theta}-T)
	\end{align*}
	for each $n \in \bbZ$. This proves the assertion.  
\end{proof}

\section{The peripheral point spectrum of Markov operators} \label{section_per_pnt_spec_markov}

In this section we focus on Markov operators on $C(K;\bbC)$-spaces (where $K$ is some compact Hausdorff space). Recall from the introduction that a bounded linear operator $T$ on $E = C(K;\bbC)$ is called a \emph{Markov operator} if $T$ is positive and $T\mathbbm{1} = \mathbbm{1}$.

\begin{theorem} \label{thm_dual_markov_op_on_c_k_has_cyclic_per_point_spec}
	Let $T$ be a Markov operator on a $C(K;\bbC)$-space. If $C(K;\bbC)$ has a pre-dual Banach lattice and $T$ has a pre-adjoint, then we have $\dim \ker(e^{i\theta}-T) \le \dim \ker(e^{in\theta}-T)$ for each $n \in \bbZ$ and each $\theta \in \bbR$. In particular, the peripheral point spectrum of $T$ is cyclic.
\end{theorem}
\begin{proof}
	This follows immediately from Theorem~\ref{thm_dominated_eigenvector_dual} if we set $x := \mathbbm{1}$.  
\end{proof}

Note that we could also interpret the above theorem as a special case of Theorem~\ref{thm_cyclcic_per_pnt_spec_ws_bounded_and_dual} since every Markov operator is power-bounded and thus (WS)-bounded.

The assertion of Theorem~\ref{thm_dual_markov_op_on_c_k_has_cyclic_per_point_spec} fails if $C(K,\bbC)$ is not assumed to have a pre-dual. This is demonstrated by the following example.

\begin{example} \label{examp_markov_op_with_non_cyclic_per_point_spec}
	There is a $C(K;\bbC)$-space $E$ and a Markov operator $T$ on $E$ such that $i$ is an eigenvalue of $T$, but $-1$ is not. \par 
	Indeed let $\overline{\mathbb{N}}_0 = \mathbb{N}_0 \cup \{\infty\}$ be the one-point compactification of the discrete space $\mathbb{N}_0$ and let $\mathbb{Z}_4 := \{\bar 0, \bar 1, \bar 2, \bar 3\}$ be endowed with the discrete topology and the addition modulo $4$. We set $K := \mathbb{Z}_4 \, \dot{\cup} \; \overline{\mathbb{N}}_0$, $E := C(K;\bbC)$ and we define an operator $T \in \calL(E)$ by
	\begin{align*}
		(Tf)(k) =
		\begin{cases}
			f(k-1) \quad & \text{if } k \in \mathbb{Z}_4 \\
			f(k-1) \quad & \text{if } k \in \mathbb{N} \cup \{\infty\} \\
			\frac{1}{2}\big(f(\bar 1) + f(\bar 3)\big) \quad & \text{if } k = 0 \text{.}
		\end{cases}
	\end{align*}
	Then $T$ is indeed a Markov operator on $E$. Note that $i$ is an eigenvalue of $T$ with a corresponding eigenfunction $g \in C(K)$ given by
	\begin{align*}
		g(k) =
		\begin{cases}
			(-i)^j \quad & \text{for } k \in \mathbb{Z}_4 \text{ and } k = \bar j \\
			0 \quad & \text{else.}
		\end{cases}
	\end{align*}
	On the other hand, $-1$ is not an eigenvalue of $T$. Indeed, assume for a contradiction that $Th = -h$ for a function $0 \not= h \in C(K)$. Then we must have $h(\bar 0) \not= 0$, since $h(\bar 0) = 0$ would imply $h(k) = 0$ for all $k \in K$. Hence, we may assume that $h(\bar 0) = 1$ and thus we obtain
	\begin{align*}
		h(\bar 0) = h(\bar 2) = 1 \quad \text{and} \quad h(\bar 1) = h(\bar 3) = -1 \text{,}
	\end{align*}
	which implies $h(0) = 1$. Therefore, $h(k) = (-1)^k$ for all $k \in \mathbb{N}_0$, which contradicts the continuity of $h$ at $\infty$. \par
\end{example}

The preceding example appeared previously in the preprint collection \cite{Gluck2014}. 

\begin{remark} \label{rem_adaption_of_markov_op_example}
	From Example~\ref{examp_markov_op_with_non_cyclic_per_point_spec} we can easily obtain another example of a contractive positive operator $T$ on an AM-space with order continuous norm such that $r(T) = 1$ but such that $\sigma_{\per,\pnt}(T)$ is not cyclic. \par 
	Indeed, let $K$, $E$ and $T$ be as in Example~\ref{examp_markov_op_with_non_cyclic_per_point_spec}. By $F$ we denote the closed subspace in $E$ of functions $f \in E$ which fulfil $f(\infty) = 0$. Then $F$ is a complex Banach lattice (in fact, an AM-space) and has order continuous norm since it is isometrically lattice isomorphic to the space $c_0(\bbN; \bbC)$ of sequences which converge to $0$. The operator $T$ leaves $F$ invariant, and the restriction $T|_F$ is a contraction which fulfils $i,-i \in \sigma_{\pnt}(T|_F)$ but $-1,1 \not\in \sigma_{\pnt}(T|_F)$. \par 
	Note that by Theorem~\ref{thm_cyclic_per_pnt_spec_ergodic_order_cont_norm} this shows that $T|_F$ is not mean ergodic (and thus $T$ cannot be mean ergodic, either).
\end{remark}

If $C(K;\bbC)$ does not have a pre-dual Banach lattice, but is still order complete (or if $C(K;\bbC)$ has a pre-dual, but $T$ does not have a pre-adjoint) we can still prove a result on the rational peripheral point spectrum (cf.~Definition~\ref{def_rational_per_point_spec}):

\begin{proposition} \label{prop_markov_op_on_order_complete_space_rational_per_pnt_spec}
	Let $E$ be an order complete $C(K;\bbC)$-space and let $T$ be a Markov operator on $E$. Then we have $\dim \ker(e^{i\theta}-T) \le \dim \ker(e^{in\theta}-T)$ for each $\theta \in 2\pi\bbQ$ and each $n \in \bbZ$. In particular, the rational peripheral point spectrum of $T$ is cyclic.
\end{proposition}
\begin{proof}
	Let $e^{i\theta}$ ($\theta \in 2\pi\bbQ$) be an eigenvalue of $T$ and choose $N \in \bbN$ such that $e^{iN\theta} = 1$. The operator $T^N$ is again a Markov operator, so its fixed space $F := \ker(1-T^N)$ is a complex Banach lattice with respect to an appropriate norm by Corollary~\ref{cor_fixed_space_of_markov_op_real}. The operator $T$ leaves $F$ invariant and $T|_F$ is periodic with period $N$. Thus, $T|_F$ is invertible with inverse $T|_F^{N-1}$. As its inverse operator is positive, $T|_F$ is a lattice isomorphism on $F$. Since the eigenspaces $\ker(e^{i\theta}-T)$ and $\ker(e^{in\theta}-T)$ ($n \in \bbZ$) are contained in $F$, we conclude by using Proposition~\ref{prop_eigenvalues_and_funtions_of_lattice_hom}(b) that
	\begin{align*}
		\dim \ker(e^{i\theta}-T) = \dim \ker(e^{i\theta}-T|_F) \le \dim \ker(e^{in\theta}-T|_F) = \dim \ker(e^{in\theta}-T)
	\end{align*}
	for each $n \in \bbZ$.  
\end{proof}

Again, it would be interesting to know whether in Proposition~\ref{prop_markov_op_on_order_complete_space_rational_per_pnt_spec} the entire peripheral point spectrum is cyclic. In any case, Example~\ref{examp_markov_op_with_non_cyclic_per_point_spec} shows that one cannot drop the condition on $C(K,\bbC)$ to be order complete in Proposition~\ref{prop_markov_op_on_order_complete_space_rational_per_pnt_spec}, even if we only consider the rational peripheral point spectrum.

\section{The peripheral spectrum of positive operators} \label{section_per_spec}

Now we use our results on the peripheral point spectrum from Section~\ref{section_per_pnt_spec_general} to analyse the peripheral spectrum of positive operators. Hardly surprising, almost all results in the current section are based on ultra power constructions. Let us fix some notation for these constructions: If $\calU$ is a free ultra filter on $\bbN$ and $X$ is a Banach space, then we denote by $X_\calU$ the ultra power of $X$ with respect to $\calU$. For a sequence $x = (x_n) \in l^\infty(\bbN;X)$ we denote by $x_\calU := (x_n)_\calU$ the equivalence class of $x$ in $X_\calU$. If $T \in \calL(X)$ then $T_\calU$ denotes the lifting of $T$ to $X_\calU$ given by $T_\calU x_\calU = (Tx_n)_\calU$ for each $x = (x_n) \in l^\infty(\bbN;X)$. \par 

In \cite[Theorem~4.7]{Lotz1968} Lotz proved that an Abel-bounded positive operator on a complex Banach lattice has cyclic peripheral spectrum (see also \cite[the first part of the proof of Theorem~V.4.9]{Schaefer1974} for an English presentation of the proof). Using the same technique, Scheffold showed in \cite[Satz~3.6]{Scheffold1971} that a positive operator $T$, $r(T) = 1$, also has cyclic peripheral spectrum if $\liminf_{n \to \infty} ||T^n|| < \infty$. In \cite[p.\,157]{Grobler1995} Grobler asked whether other boundedness conditions might be sufficient to prove the same result. We show now that this is possible for (WS)-bounded operators:

\begin{theorem} \label{thm_ws_bounded_op_zykl_per_spec}
	Let $T$ be a positive operator on a complex Banach lattice $E$, $r(T) = 1$. If $T$ is (WS)-bounded, then the peripheral spectrum of $T$ is cyclic.
\end{theorem}

For the proof of the theorem we do not use the technique developed by Lotz in \cite[Section~4]{Lotz1968} (in fact, it is not clear to the author whether Theorem~\ref{thm_ws_bounded_op_zykl_per_spec} can be proved by this technique), but we rather employ Theorem~\ref{thm_cyclcic_per_pnt_spec_ws_bounded_and_dual} together with the following duality result on ultra powers. 

\begin{lemma} \label{lem_ultra_power_dual_embedding}
	Let $E$ be a complex Banach lattice and let $\calU$ be a free ultra filter on $\bbN$.
	\begin{enumerate}[(a)]
		\item The canonical embedding
			\begin{align*}
				j: (E')_\calU \to (E_\calU)' \text{,} \quad \langle j((x'_n)_\calU), (x_n)_\calU \rangle = \lim_\calU \langle x'_n,x_n\rangle
			\end{align*}
			is isometric and positive. \par 
		\item We have $(T_\calU)' \circ j = j \circ (T')_\calU$. \par 
		\item In particular, the peripheral spectrum of $(T_\calU)'$ consists of eigenvalues of $(T_\calU)'$.
	\end{enumerate}
\end{lemma}
\begin{proof}
	(a) It is well-known that $j$ is isometric (see \cite[the beginning of Section~7]{Heinrich1980}) and clearly, the mapping is also positive. \par 
	(b) This is a straightforward exercise which we leave to the reader. \par 
	(c) We have $\sigma_{\per}((T_\calU)') = \sigma_{\per}(T') \subset \sigma_{\pnt}((T')_\calU) \subset \sigma_{\pnt}((T_\calU)')$, where the last inclusion follows from (b).  
\end{proof}

\begin{proof}[Proof of Theorem~\ref{thm_ws_bounded_op_zykl_per_spec}]
	Fix a free ultra filter $\calU$ on $\bbN$. Lemma~\ref{lem_ultra_power_dual_embedding}(c) implies that $\sigma_{\per}(T) = \sigma_{\per,\pnt}((T_\calU)')$. The operator $(T_\calU)'$ is (WS)-bounded since $||f((T_\calU)')|| = ||(f(T)_\calU)'|| = ||f(T)||$ for every analytic function $f$ on $\overline{\bbD}$. Hence, Theorem~\ref{thm_cyclcic_per_pnt_spec_ws_bounded_and_dual} yields that the peripheral point spectrum of $(T_\calU)'$ is cyclic.  
\end{proof}

We note that, instead of using Theorem~\ref{thm_cyclcic_per_pnt_spec_ws_bounded_and_dual} (respectively the underlying Theorem~\ref{thm_dominated_eigenvector_dual}), we could have also employed a somewhat different result of Krieger \cite[Satz~2.2.2]{Krieger1969} in the proof of Theorem~\ref{thm_ws_bounded_op_zykl_per_spec}. \par 

Next we give some cyclicity results on spectral values of positive operators which are based on the existence of certain approximate eigenvectors. Let us recall the relevant notions. If $X$ is a complex Banach space, $T \in \calL(X)$ and $\lambda \in \sigma(T)$ then we call a sequence $(x_n) \subset X$ an \emph{approximate eigenvector} of $T$ for the spectral value $\lambda$ if $(x_n)$ is bounded, if $0 < \liminf_n ||x_n||$ and if $(\lambda - T)x_n \to 0$. If such an approximate eigenvector exists, then $\lambda$ is called an \emph{approximate eigenvalue} of $T$.

\begin{definition} \label{def_dominated_approx_eigenvector_condition}
	Let $T$ be a positive operator on a complex Banach lattice $E$ and let $re^{i\theta}$ ($r>0$, $\theta \in \bbR$) be an approximate eigenvalue of $T$. We say that $re^{i\theta}$ fulfils the \emph{dominated approximate eigenvector condition} if $r$ is also an approximate eigenvalue of $T$ and if there are approximate eigenvectors $(z_n)$ and $(x_n)$ of $T$ for the spectral values $re^{i\theta}$ and $r$, respectively, such that $|z_n| \le x_n$ for all $n \in \bbN$. 
\end{definition}

Note that the condition that $re^{i\theta}$ and $r$ be approximate eigenvalues is always fulfilled if $r = r(T)$, i.e.~if $re^{i\theta}$ is a peripheral spectral value of $T$. \par 

For peripheral spectral values, the main idea of the following results is already implicitly contained in the proof of \cite[Folgerung~2.2.3]{Krieger1969}. However, our results also hold for spectral values $re^{i\theta}$ with $r < r(T)$.

\begin{theorem} \label{thm_dae_and_pre_adjoint_imply_powers_in_spec}
	Let $T$ be a positive operator on a complex Banach lattice $E$; suppose that $E$ has a pre-dual Banach lattice and $T$ has a pre-adjoint. If an approximate eigenvalue $re^{i\theta}$ of $T$ ($r > 0$, $\theta \in \bbR$) fulfils the dominated approximate eigenvector condition, then $re^{in\theta} \in \sigma(T)$ for all $n \in \bbZ$.
\end{theorem}
\begin{proof}
	We may assume that $r = 1$. Let $F$ be a pre-dual Banach lattice of $E$, let $S$ be the pre-adjoint of $T$ and fix a free ultra filter $\calU$ on $\bbN$. \par 
	By assumption, we find approximate eigenvectors $z = (z_n)$ and $x = (x_n)$ of $T$ for the spectral values $e^{i\theta}$ and $1$ such that $|z_n| \le |x_n|$ for every $n \in \bbN$. Hence, $z_\calU$ and $x_\calU$ are eigenvectors of $T_\calU$ for the eigenvalues $e^{i\theta}$ and $1$ which fulfil $|z_\calU| \le x_\calU$. \par 
	Now, let $j: E_\calU \to (F_\calU)'$ be the canonical embedding from Lemma~\ref{lem_ultra_power_dual_embedding}(a). Then, by Lemma~\ref{lem_ultra_power_dual_embedding}(b), $j(z_\calU)$ and $j(x_\calU)$ are eigenvectors of $(S_\calU)'$ for the eigenvalues $e^{i\theta}$ and $1$, respectively. Moreover, the positivity of $j$ implies $|j(z_\calU)| \le j(|z_\calU|) \le j(x_\calU)$; hence, we can apply Theorem~\ref{thm_dominated_eigenvector_dual} to the operator $(S_\calU)'$ to conclude that $e^{in\theta} \in \sigma((S_\calU)') = \sigma(T)$ for each $n \in \bbZ$.  
\end{proof}

\begin{corollary} \label{cor_dae_on_refl_space_implies_powers_in_spec}
	Let $T$ be a positive operator on a reflexive complex Banach lattice $E$. If an approximate eigenvalue $re^{i\theta}$ of $T$ ($r>0$, $\theta \in \bbR$) fulfils the dominated approximate eigenvector condition, then $re^{in\theta} \in \sigma(T)$ for all $n \in \bbZ$.
\end{corollary}

\begin{corollary} \label{cor_dae_for_ajdoint_implies_powers_in_spec_of_original_op}
	Let $T$ be a positive operator on a Banach lattice $E$. If a spectral value $re^{i\theta}$ of $T$ ($r>0$, $\theta \in \bbR$) is an approximate eigenvalue of the adjoint $T'$ and fulfils the dominated approximate eigenvector condition for $T'$, then $re^{in\theta} \in \sigma(T)$ for all $n \in \bbZ$. 
\end{corollary}

We want to point out that, even in finite dimensions, a peripheral spectral value of a positive operator need not fulfil the dominated approximate eigenvector condition, in general. This is shown by the next example.

\begin{example} \label{examp_no_dae}
	There is a positive operator $T$ on $\bbC^4$, $r(T) = 1$, such that $-1$ is a spectral value of $T$ which does neither fulfil the dominated approximate eigenvector condition for $T$ nor for $T'$. \par 
	Indeed, let $T$ be the operator on $\bbC^4$ whose representation matrix with respect to the canonical basis is given by
	\begin{align*}
		T =
		\begin{pmatrix}
			0 & 1 & 0 & 1 \\
			1 & 0 & 0 & 0 \\
			1 & 0 & 1 & 0 \\
			0 & 0 & 0 & 1
		\end{pmatrix} \text{.}
	\end{align*}
	A short computation shows that $\sigma(T) = \{-1,1\}$ and that the eigenspaces $\ker(-1-T)$ and $\ker(1-T)$ are spanned by $v_1 = (2,-2,-1,0)$ and $v_2 = (0,0,1,0)$, respectively. \par 
	Now assume for a contradiction that $-1$ fulfils the dominated approximate eigenvector condition for $T$. Then we find approximate eigenvectors $(z_n)$ for $-1$ and $(x_n)$ for $1$ such that $|z_n| \le x_n$. After choosing an appropriate subsequence twice, we may assume that $z_n \to z \not= 0$ and $x_n \to x > 0$. We have $|z| \le x$, and $z$ and $x$ are eigenvectors of $T$ for the eigenvalues $-1$ and $1$, respectively. Therefore, $z = \alpha v_1$ for some $\alpha \in \bbC \setminus\{0\}$ and $x = \beta v_2$ for some $\beta > 0$. However, looking at $v_1$ and $v_2$, we see that this contradicts $|z| \le x$. \par 
	As similar argument shows that $-1$ does not fulfil the the dominated approximate eigenvector condition for $T'$; just observe that the eigenspaces $\ker(-1-T')$ and $\ker(1-T')$ are spanned by $(2,-2,0,-1)$ and $(0,0,0,1)$, respectively.
\end{example}

Let us finally present a situation in which the dominated approximate eigenvector condition is automatically fulfilled so that our above results can be applied. Although the following theorem is already known from \cite[Folgerung~2.2.3]{Krieger1969} we chose to include it here since, to the knowledge of the author, the theorem has never appeared in English, and since our proof is partly different from the proof given in \cite[Folgerung~2.2.3]{Krieger1969}.

\begin{theorem} \label{thm_resolvent_grows_fast}
	Let $T$ be a positive operator on a complex Banach lattice $E$, $r(T) = 1$. Let $e^{i\theta}$ ($\theta \in \bbR$) be a peripheral spectral value of $T$. If
	\begin{align*}
		\limsup_{r \downarrow 1} \frac{||R(re^{i\theta},T)||}{||R(r,T)||} > 0 \text{,}
	\end{align*}
	then $e^{in\theta} \in \sigma(T)$ for each $n \in \bbZ$.
\end{theorem}
\begin{proof}
	Let $e^{i\theta}$ ($\theta \in \bbR$) be a peripheral spectral value of $T$; we show that $e^{i\theta}$ fulfils the dominated approximate eigenvector condition for $T'$, so that we can apply Corollary~\ref{cor_dae_for_ajdoint_implies_powers_in_spec_of_original_op}. \par 
	To this end, let $c > 0$ and let $1 < r_n \downarrow 1$ such that $||R(r_ne^{i\theta},T')|| \ge c\,||R(r_n,T')||$ for each $n$. Choose a sequence of normalized functionals $(a'_n) \subset X'$ which fulfil the estimate $||R(r_ne^{i\theta},T')a'_n|| \ge \frac{1}{2}||R(r_ne^{i\theta},T')||$ for each $n \in \bbN$. Now, define
	\begin{align*}
		z'_n := \frac{R(r_ne^{i\theta},T')a'_n}{||R(r_ne^{i\theta},T')a'_n||} \quad \text{and} \quad x'_n := \frac{R(r_n,T')|a'_n|}{||R(r_ne^{i\theta},T')a'_n||}
	\end{align*}
	Since $||R(r_ne^{i\theta},T')a'_n|| \to \infty$ as $n \to \infty$, we know that $(z'_n)$ is an approximate eigenvector of $T'$ for the spectral value $e^{i\theta}$. Moreover, we have $|z'_n| \le x'_n$ for each $n$. Furthermore, it is clear that $||x'_n|| \ge 1$ for each $n$, and on the other hand we obtain the estimate
	\begin{align*}
		||x'_n|| \le \frac{2\,||R(r_n,T')||}{||R(r_ne^{i\theta},T')||} \le \frac{2}{c}
	\end{align*}
	for each $n$. Hence, the sequence $(x'_n)$ is bounded and therefore it is an approximate eigenvector of $T'$ for the spectral value $1$. This proves that $e^{i\theta}$ fulfils indeed the dominated approximate eigenvector condition for $T'$.  
\end{proof}

Let us briefly compare our proof of the above result with the one given by Krieger in \cite[Folgerung~2.2.3]{Krieger1969}: We use the same construction as Krieger to obtain approximate eigenvectors which fulfil a certain domination condition. Also rather similar to Krieger we then employ an ultra power construction to consider the approximate eigenvectors as eigenvectors (cf.~our proof of Theorem~\ref{thm_dae_and_pre_adjoint_imply_powers_in_spec}). The remainder of the proof however differs from Krieger's approach: While the rest of our argument relies on Theorem~\ref{thm_dominated_eigenvector_dual}, Krieger uses a different spectral result \cite[Satz~2.2.2]{Krieger1969} which is based on a rather technical resolvent estimate.

\section{On the boundary point spectrum of strongly continuous Markov semigroups} \label{section_boundary_pnt_spec_of_markov_sg}

In this section we want to show that a similar construction as in Example~\ref{examp_markov_op_with_non_cyclic_per_point_spec} is also possible for $C_0$-semigroups of Markov operators. Let us briefly fix some notation: Consider a $C_0$-semigroup $(T(t))_{t \ge 0}$ on a complex Banach space $X$ with generator $A$. If $s(A)$ denotes the \emph{spectral bound} of $A$ and if $\sigma(A) \not= \emptyset$, then the set $\sigma_{\operatorname{bnd},\pnt}(A) := \sigma(A) \cap (s(A) + i\bbR)$ is called the \emph{boundary point spectrum} of $A$. A subset $M \subset \bbC$ is called \emph{additively cyclic} if for each $\alpha + i \beta \in M$ ($\alpha, \beta \in \bbR$) we also have $\alpha + in\beta \in M$ for each $n \in \bbZ$. The boundary point spectrum of the generator of a $C_0$-semigroups of positive operators on a complex Banach lattice $E$ need not be additively cyclic in general (see e.g.~\cite[Examples~B-III.2.13 and C-III.4.4]{Arendt1986} for some counter examples), but it is under some additional assumptions (see for example \cite[Corollary~C-III.2.3, Section~C-III.3 and Corollary~C-III.4.3]{Arendt1986}). \par 
If $E = C(K;\bbC)$ for some compact Hausdorff space $K$, then we call a $C_0$-semigroup $(T(t))_{t \ge 0}$ on $E$ a \emph{Markov semigroup} if each operator $T(t)$ is a Markov operator. One might ask whether it is possible to prove similar results as in Theorem~\ref{thm_dual_markov_op_on_c_k_has_cyclic_per_point_spec} or in Proposition~\ref{prop_markov_op_on_order_complete_space_rational_per_pnt_spec} also for Markov semigroups. Let us explain briefly why this question does not really make sense for $C_0$-semigroups:

\begin{remark} \label{rem_marko_sg_on_order_complete_space}
	Let $K$ be a compact Hausdorff space such that $C(K;\bbC)$ is order complete and let $(T(t))_{t \ge 0}$ be a $C_0$-semigroup of Markov operators on $C(K;\bbC)$ with generator $A$. Then it follows from \cite[Theorem A-II.3.6]{Arendt1986} that the operator $A$ is automatically bounded. Moreover, the boundary spectrum $\sigma_{\operatorname{bnd}}(A) := \sigma(A) \cap (s(A) +i \bbR) = \sigma(A) \cap i\bbR$ of $A$ is additively cyclic due to \cite[Proposition C-III.2.9 and Theorem C-III.2.10]{Arendt1986}, so we conclude that in fact $\sigma_{\operatorname{bnd}}(A) = \sigma_{\operatorname{bnd},\pnt}(A) = \{0\}$.
\end{remark}

More generally, as in the single operator case, one might wonder whether the generator of each $C_0$-semigroup of Markov operators on an arbitrary $C(K;\bbC)$-space has additively cyclic boundary point spectrum. As in the single operator case, the answer is negative. This is demonstrated by the following adaptation of Example~\ref{examp_markov_op_with_non_cyclic_per_point_spec}. As Example~\ref{examp_markov_op_with_non_cyclic_per_point_spec} it appeared previously in the preprint collection \cite{Gluck2014}.

\begin{example} \label{examp_markov_sg_with_non_cyclic_per_pnt_spec}
	There is a $C(K;\bbC)$-space $E$ and a strongly continuous Markov semigroup $(T(t))_{t \ge 0}$ on $E$ such that $i$ is an eigenvalue of its generator $A$, but $2i$ is not. \par 
	To construct such an example, let $(R(t))_{t \ge 0}$ be the rotation semigroup on $C(\bbT;\bbC)$, given by $R(t)f(x) = f(e^{-it}x)$ for each $f \in C(\bbT;\bbC)$ and each $x \in \bbT$. Now, define $K := \mathbb{T} \, \dot \cup \, [0,\infty]$ and $E = C(K;\bbC)$. For each $f \in E$ and each $t \ge 0$ define an operator $T(t) \in \calL(E)$ by
	\begin{align*}
		T(t)f(x) = 
		\begin{cases}
			f(e^{-it}x) \quad & \text{if } x \in \mathbb{T} \\
			f(x-t) \quad & \text{if } x \in [t,\infty] \\
			e^{-(t-x)}f(0) + e^{-(t-x)}\int_0^{t-x} e^s \, \langle \mu, R(s)f|_\mathbb{T} \rangle \, ds \quad & \text{if } x \in [0,t) \text{,}
		\end{cases}
	\end{align*}
	where $\mu \in C(\mathbb{T}; \bbC)'$ is the functional on $C(\mathbb{T};\bbC)$ that is defined by $\langle \mu, f \rangle = \frac{1}{2}\big( f(i) + f(-i) \big)$. \par 
	Now, let $C^1([0,\infty];\bbC)$ be the set of all functions in $C([0,\infty];\bbC)$ such that $f|_{[0,\infty)}$ is continuously differentiable and such that its derivative has a continuous extension to $[0,\infty]$. A somewhat lengthy, but straight forward computation shows that $(T(t))_{t \ge 0}$ is a $C_0$-semigroup of Markov operators on $E$, that the domain of its generator $A$ is given by
	\begin{align*}
		D(A) = \{f \in C(K): \; & f|_\mathbb{T} \in C^1(\mathbb{T}) \text{, } \, f|_{[0,\infty]} \in C^1([0,\infty]) \text{,} \\
		& f'(0) = f(0) - \langle \mu, f|_\mathbb{T} \rangle\}
	\end{align*}
	and that we have
	\begin{align*}
		Af(x) = 
		\begin{cases}
			-\frac{d}{d\theta} f(xe^{i\theta})|_{\theta = 0} \quad & \text{if } x \in \mathbb{T} \\
			-f'(x) \quad & \text{if } x \in [0,\infty)
		\end{cases}
	\end{align*}
	for all $f \in D(A)$. Now we can immediately check that $i$ is an eigenvalue of $A$; indeed, a corresponding eigenfunction is given by
	\begin{align*}
		g(x) =
		\begin{cases}
			x^{-1} \quad & \text{if } x \in \mathbb{T} \\
			0 \quad & \text{if } x \in [0,\infty] \text{.}
		\end{cases}
	\end{align*}
	To show that $2i$ is not an eigenvalue of $A$, assume for a contradiction that $Ah = 2ih$ for a function $0 \not= h \in D(A)$. Then there are scalars $a,b \in \mathbb{C}$ such that
	\begin{align*}
		h(x) = 
		\begin{cases}
			a x^{-2} \quad & \text{if } x \in \mathbb{T} \\
			b e^{-2i x} & \text{if } x \in [0,\infty) \text{.}
		\end{cases}
	\end{align*}
	Since $h$ must be continuous at $\infty$, we conclude that $b = 0$. Now it follows from the equation $h'(0) = h(0) - \langle \mu, h|_\mathbb{T} \rangle$ that $\langle \mu, h|_\mathbb{T} \rangle = 0$. Since $\langle \mu, h|_\mathbb{T} \rangle = \frac{1}{2}\big( h(i) + h(-i) \big) = -a$, we also have $a = 0$. This contradicts $h \not= 0$.
\end{example}

\appendix

\section{The signum operator} \label{section_the_signum_operator}

In this appendix we shortly recall some facts about the signum operator on complex Banach lattices which are needed in the article. First we recall the following result from \cite[Section~C-I.8]{Arendt1986}:

\begin{proposition} \label{prop_signum_op_on_principal_ideal}
	Let $E$ be a complex Banach lattice and let $f \in E \setminus \{0\}$. Then there exists a unique linear operator $S_f$ on $E_{|f|}$ which fulfils the following two conditions:
	\begin{enumerate}[(a)]
		\item $S_f \overline{f} = |f|$, where $\overline{f}$ denotes the complex conjugate vector of $f$. \par 
		\item $|S_fg| \le |g|$ for all $g \in E_{|f|}$.
	\end{enumerate}
\end{proposition}

The operator $S_f$ is called the \emph{signum operator} associated to $E_f$. If we identify $E_{|f|}$ with a $C(K;\bbC)$-space ($K$ compact) by means of the Kakutani representation theorem such that $|f|$ corresponds to the constant $1$-function on $K$, then we have $S_{f}g = fg$ for each $g \in C(K;\bbC)$ where the multiplication is computed in $C(K;\bbC)$. Now we come to the major definition in this appendix (compare \cite[Definitions~B-III.2.2(b) and C-III.2.1]{Arendt1986}).

\begin{definition} \label{def_lattice_powers_of_a_function}
	Let $E$ be a complex Banach lattice and let $f \in E \setminus \{0\}$. By means of the Kakutani representation theorem we can identify the principal ideal $E_{|f|}$ with a $C(K;\bbC)$-space for some compact Hausdorff-space $K$ such that $|f|$ corresponds to the constant $1$-function on $K$. Using the multiplication on $C(K;\bbC)$, we define $f^{[n]} := f^n$ for each $n \in \bbZ$.
\end{definition}

Note that $f^{[n]} = S_f^n |f|$ whenever $n \ge 0$ and $f^{[n]} = S_{\overline{f}}^n |f|$ whenever $n < 0$. This shows that the definition of $f^{[n]}$ is independent of the choice of the representation $E \xrightarrow{\raisebox{-0.6 em}{\smash{\ensuremath{\sim}}}} C(K;\bbC)$. The following property of $f^{[n]}$ is important:

\begin{proposition} \label{prop_consistency_signum}
	Let $E$ be a complex Banach lattice and let $f \in E \setminus \{0\}$. Suppose that $h \in E_+$ with $f \in E_h$ and identify the principal ideal $E_h$ with a $C(K;\bbC)$-space for some compact Hausdorff space $K$, where $h$ corresponds to the constant $1$-function on $K$. In the space $C(K; \bbC)$ the vectors $f^{[n]}$ are given by
	\begin{align}
		f^{[n]}(x)=
		\begin{cases}
			(\frac{f(x)}{|f(x)|})^n |f(x)| \quad & \text{if } f(x) \not= 0 \\
			0 \quad & \text{if } f(x) = 0 \text{.}
		\end{cases}
		\label{form_lattice_powers_consistent}
	\end{align}
\end{proposition}
\begin{proof}
	First, let $n \in \bbN_0$. Define an operator $\tilde S_{f}$ on $E_{|f|} = C(K;\bbC)_{|f|}$ which is given by 
	\begin{align*}
		\tilde S_{f} g(x) =
		\begin{cases}
			\frac{f(x)}{|f(x)|} g(x) \quad & \text{if } f(x) \not= 0 \\
			0 \quad & \text{if } f(x) = 0
		\end{cases}
	\end{align*}
	for every $g \in C(K;\bbC)_{|f|}$ and every $x \in K$. Then $\tilde S_f$ fulfils properties (a) and (b) from Proposition~\ref{prop_signum_op_on_principal_ideal} and we thus have $\tilde S_f = S_f$. This implies $f^{[n]} = S_f^n |f| = \tilde S_f^n |f|$, which proves the assertion. For $n < 0$ one argues similarly, using the operators $\tilde S_{\overline{f}}$ and $S_{\overline{f}}$ instead.  
\end{proof}

We can now prove the following lemma which is a slight modification of \cite[Lemma~C-III.3.11]{Arendt1986}.

\begin{lemma} \label{lem_dimension_estimate_complex_powers}
	Let $E$ be a complex Banach lattice and let $G,H \subset E$ be to vector subspaces of $E$. Let $n \in \bbZ$ and assume that $f^{[n]} \in H$ for each $f \in G \setminus \{0\}$. Then $\dim G \le \dim H$.
	\begin{proof}
		The proof is very similar to the proof of \cite[Lemma~C-III.3.11]{Arendt1986}; for the convenience of the reader, we include it here. Let $0 < m \le \dim G$, let $g_1,...,g_m$ be linearly independent elements of $G$ and define $u := |g_1| + ... + |g_m|$. Then we can identify the principal ideal $E_u$ with a $C(K; \bbC)$-space. There are points $x_1,...,x_m \in K$ and functions $f_1,...,f_m \in C(K;\bbC)$ with the same linear span as $g_1,...,g_m$ which have the property that $f_j(x_k) = \delta_{jk}$ (where $\delta_{jk}$ is the Dirac delta) for all $j,k \in \{1,...,m\}$; this can easily be seen by an induction over $m$. By our assumption, we have $f_j^{[n]} \in H$ for all $j \in \{1,...,m\}$ and due to Proposition~\ref{prop_consistency_signum} we can compute $f_j^{[n]}$ in $C(K;\bbC)$ by means of formula (\ref{form_lattice_powers_consistent}). Hence, we also have $f_j^{[n]}(x_k) = \delta_{jk}$ for all $j,k \in \{1,...,m\}$ and therefore, the vectors $f_1^{[n]},...,f_m^{[n]}$ are linearly independent. Thus $m \le \dim H$, which proves the assertion.  
	\end{proof}
\end{lemma}

\subsubsection*{Acknowledgements}
	I would like to thank Manuel Bernhard and Manfred Sauter for their help in the construction of Example~\ref{examp_no_dae}; moreover, Manfred Sauter assisted me with the proof of Lemma~\ref{lem_dimension_estimate_complex_powers}. My thanks also go to Rainer Nagel who suggested the investigation of weakly almost periodic operators in the context of Theorem~\ref{thm_cyclic_per_pnt_spec_almost_weakly_periodic}.

\bibliographystyle{spmpsci}
\bibliography{literature}

\end{document}